\documentclass{amsart}

\usepackage{amsmath,latexsym,amssymb}
\usepackage{color}
\usepackage{dsfont}
\usepackage{mathrsfs}
\usepackage{tikz}
\usetikzlibrary{arrows,decorations.markings,fit,matrix,positioning,calc,shapes.symbols,shapes}

\address[jim.conant@gmail.com]{Jim Conant, Department of Mathematics, University of Tennessee, Knoxville, TN, USA}

\definecolor{dmagenta}{rgb}{.5,0,.5} 
\definecolor{dred}{rgb}{.5,0,0} 
\definecolor{dgreen}{rgb}{0,.5,0} 
\definecolor{blue}{rgb}{0,0,0.5} 
\definecolor{black}{rgb}{0,0,0} 
\definecolor{vdgreen}{rgb}{0,.3,0} 
\definecolor{vdred}{rgb}{.3,0,0} 
\definecolor{red}{rgb}{1,0,0} 

\newcommand\cG{{\mathcal{G}}}

\newcommand\cO{{\mathcal{O}}} %Notation for a cyclic operad
\newcommand{\Lie}{\mathsf{Lie}}  %Lie operad
\newcommand{\F}{{\mathds{k}}}
\newcommand{\Z}{{\mathbb{Z}}}
\newcommand{\C}{\F}

\newcommand{\ext}{\bigwedge\nolimits}

\newcommand{\id}{\mathrm{Id}}   %Identity
\DeclareMathOperator{\Tr}{Tr}  % Trace map
\DeclareMathOperator{\SP}{Sp}  % symplectic group
\DeclareMathOperator{\Out}{Out} % Out

\DeclareMathOperator{\sym}{Sym}
\DeclareMathOperator{\GL}{GL}  % general linear group
\DeclareMathOperator{\Aut}{Aut}  % Aut
\DeclareMathOperator{\Mod}{Mod}  % Mapping class groups
\DeclareMathOperator{\im}{im}  % image
\DeclareMathOperator{\SL}{SL}

\let\ker\undefined
\DeclareMathOperator{\ker}{ker}  % kernel
\newcommand{\SF}[1]{{\mathbb S}_{#1}}  % Schur functor
\newcommand{\arity}[1]{{(\!(#1)\!)}}

\newcommand{\wt}{\widetilde}
\newcommand{\wh}{\widehat}
\newcommand{\PSL}{\operatorname{PSL}}
\newcommand{\smat}{s}
\newcommand{\tmat}{t}

\newcommand{\homfunctor}{\mathscr{H}}

\newcommand{\graphA}[3]{
\begin{minipage}{3cm}
\resizebox{3cm}{!}{
\begin{tikzpicture}[scale=.8]
\coordinate(a) at (0,0);
\coordinate(b) at (1.5,0);
\coordinate(c) at (3.8,0);
\coordinate(d) at (5,0);
\coordinate(e) at (2,-1);
\coordinate(f) at (5,-.25);
\coordinate(g) at (1,1.5);
\coordinate(h) at (2.5,1.0);
\coordinate(i) at (3.8,.25);
\coordinate(j) at (5,.25);
\node[hopf, at=(g)] (G){${#1}$};
\node[hopf, at=(h)] (H){${#2}$};
\node[hopf, at=(e)] (E){${#3}$};
\draw[thick] (a) to (d);
\draw[thick] (a) to[out=90, in=180] (G);
\draw[thick] (b) to[out=90, in=180] (H);
\draw[thick] (a) to[out=270, in =180] (E);
\draw[thick] (c) to (i);
\draw[thick] (d) to (j);
\draw[thick] (d) to (f);
\begin{scope}[decoration={markings,mark = at position 0.5 with {\arrow{stealth}}}]
\draw[densely dashed, postaction=decorate] (E) to[densely dashed, postaction=decorate, out=0, in=270](f);
\draw[densely dashed, postaction=decorate] (G) to[densely dashed, postaction=decorate,out=0,in=90](i);
\draw[densely dashed, postaction=decorate] (H) to[densely dashed, postaction=decorate,out=0,in=90](j);
\end{scope}
\end{tikzpicture}
}
\end{minipage}
}
\newcommand{\graphB}[3]{
\begin{minipage}{3cm}
\resizebox{3cm}{!}{
\begin{tikzpicture}[scale=.8]
\coordinate(a) at (0,0);
\coordinate(b) at (1.5,0);
\coordinate(c) at (3.5,0);
\coordinate(d) at (5,0);
\coordinate(e) at (2,-1);
\coordinate(f) at (5,-.25);
\coordinate(g) at (1,1.5);
\coordinate(h) at (2.5,1.0);
\coordinate(i) at (3.5,.25);
\coordinate(j) at (5,.25);
\node[hopf, at=(g)] (G){$#1$};
\node[hopf, at=(h)] (H){$#2$};
\node[hopf, at=(e)] (E){$#3$};
\draw[thick] (a) to (d);
\draw[thick] (a) to[out=90, in=180] (G);
\draw[thick] (b) to[out=90, in=180] (H);
\draw[thick] (a) to[out=270, in =180] (E);
\draw[thick] (c) to (i);
\draw[thick] (d) to (j);
\draw[thick] (d) to (f);
\begin{scope}[decoration={markings,mark = at position 0.5 with {\arrow{stealth}}}]
\draw[densely dashed, postaction=decorate] (E) to[densely dashed, postaction=decorate, out=0, in=270](f);
\draw[densely dashed, postaction=decorate] (G) to[densely dashed, postaction=decorate,out=0,in=90](j);
\draw[densely dashed, postaction=decorate] (H) to[densely dashed, postaction=decorate,out=0,in=90](i);
\end{scope}
\end{tikzpicture}
}
\end{minipage}
}
\newcommand{\graphC}[3]{
\begin{minipage}{3cm}
\resizebox{3cm}{!}{
\begin{tikzpicture}[scale=.8]
\coordinate(a) at (0,0);
\coordinate(b) at (1.8,0);
\coordinate(c) at (2.7,0);
\coordinate(d) at (4.5,0);
\coordinate(e) at (1,-1);
\coordinate(f) at (4.5,-.25);
\coordinate(g) at (.7,1);
\coordinate(h) at (3.7,1);
\coordinate(i) at (1.8,.25);
\coordinate(j) at (4.5,.25);
\node[hopf, at=(g)] (G){$#1$};
\node[hopf, at=(h)] (H){$#2$};
\node[hopf, at=(e)] (E){$#3$};
\draw[thick] (a) to (d);
\draw[thick] (a) to[out=90, in=180] (G);
\draw[thick] (c) to[out=90, in=180] (H);
\draw[thick] (a) to[out=270, in =180] (E);
\draw[thick] (b) to (i);
\draw[thick] (d) to (j);
\draw[thick] (d) to (f);
\begin{scope}[decoration={markings,mark = at position 0.5 with {\arrow{stealth}}}]
\draw[densely dashed, postaction=decorate] (E) to[densely dashed, postaction=decorate, out=0, in=270](f);
\draw[densely dashed, postaction=decorate] (G) to[densely dashed, postaction=decorate,out=0,in=90](i);
\draw[densely dashed, postaction=decorate] (H) to[densely dashed, postaction=decorate,out=0,in=90](j);
\end{scope}
\end{tikzpicture}
}
\end{minipage}
}

\newcommand{\boxtensor}{
\draw[fill=blue!20] (-.5,0) to (2.5,0) to (2.5,.5) to (-.5,.5) to (-.5,0);
}
\newcommand{\varboxtensor}{
\draw[fill=blue!20] (-.25,0) to (1.25,0) to (1.25,.25) to (-.25,.25) to (-.25,0);
}
\newcommand{\varid}[1]{
\draw (0,1-#1) to (0,-#1);
\draw (1,1-#1) to (1,-#1);
}
\newcommand{\vartrans}[1]{
\draw (0,1-#1) to[out=-90,in=90] (1,-#1);
\draw (1,1-#1) to[out=-90,in=90]  (0,-#1);
}

\newcommand{\roundtensor}{
\draw[fill=yellow!20] (-.5,0) to (2.5,0) to[out=90,in=0] (1,.75) to[out=180,in=90] (-.5,0);
}
\newcommand{\tritensor}{
\draw[fill=red!20] (-.5,0) to (2.5,0) to (1,1) to (-.5,0);
}
\newcommand{\transA}[1]{
\draw (0,1-#1) to[out=-90,in=90] (1,0-#1);
\draw (1,1-#1) to[out=-90,in=90] (0,0-#1);
\draw (2,1-#1) to (2,0-#1);
}
\newcommand{\transB}[1]{
\draw (0,1-#1) to[out=-90,in=90] (2,0-#1);
\draw (1,1-#1) to[out=-90,in=90] (1,0-#1);
\draw (2,1-#1) to[out=-90,in=90] (0,0-#1);
}
\newcommand{\transC}[1]{
\draw (0,1-#1) to[out=-90,in=90] (0,0-#1);
\draw (1,1-#1) to[out=-90,in=90] (2,0-#1);
\draw (2,1-#1) to[out=-90,in=90] (1,0-#1);
}
\newcommand{\transD}[1]{
\draw (0,1-#1) to[out=-90,in=90] (2,0-#1);
\draw (1,1-#1) to[out=-90,in=90] (0,0-#1);
\draw (2,1-#1) to[out=-90,in=90] (1,0-#1);
}
\newcommand{\transE}[1]{
\draw (0,1-#1) to[out=-90,in=90] (1,0-#1);
\draw (1,1-#1) to[out=-90,in=90] (2,0-#1);
\draw (2,1-#1) to[out=-90,in=90] (0,0-#1);
}
\newcommand{\transF}[1]{
\draw (0,1-#1) to[out=-90,in=90] (0,0-#1);
\draw (1,1-#1) to[out=-90,in=90] (1,0-#1);
\draw (2,1-#1) to[out=-90,in=90] (2,0-#1);
}
\newcommand{\antip}[4]{
\coordinate (a) at (0,-#4+.5);
\coordinate (b) at (1,-#4+.5);
\coordinate (c) at (2,-#4+.5);
\ifnum1=#1
\node[antipode, at =(a)] (A) {$S$};
\else
\coordinate (A) at (0,-#4+.5);
\fi
\ifnum1=#2
\node[antipode, at =(b)] (B) {$S$};
\else
\coordinate (B) at (1,-#4+.5);
\fi
\ifnum1=#3
\node[antipode, at =(c)] (C) {$S$};
\else
\coordinate (C) at (2,-#4+.5);
\fi
\draw (0,1-#4) to (A);
\draw (A) to (0,-#4);
\draw (1,1-#4) to (B);
\draw (B) to (1,-#4);
\draw (2,1-#4) to (C);
\draw (C) to (2,-#4);
}
\newcommand{\varantip}[5]{
\coordinate (a) at (0,-#4+.5);
\coordinate (b) at (1,-#4+.5);
\coordinate (c) at (2,-#4+.5);
\ifnum1=#1
\node[rec, at =(a)] (A) {$#5$};
\else
\coordinate (A) at (0,-#4+.5);
\fi
\ifnum1=#2
\node[rec, at =(b)] (B) {$#5$};
\else
\coordinate (B) at (1,-#4+.5);
\fi
\ifnum1=#3
\node[rec, at =(c)] (C) {$#5$};
\else
\coordinate (C) at (2,-#4+.5);
\fi
\draw (0,1-#4) to (A);
\draw (A) to (0,-#4);
\draw (1,1-#4) to (B);
\draw (B) to (1,-#4);
\draw (2,1-#4) to (C);
\draw (C) to (2,-#4);
}

\newcommand{\opT}[3]{
\draw(0,1-#3) to (0,-#3);
\draw(1,1-#3) to (1,-#3);
\draw(2,1-#3) to (2,-#3);
\coordinate (a) at (.5,.5-#3);
\ifnum1=#2 
\node[antipode, at=(a)] (A) {$S$};
\else
\coordinate (A) at (.5,.5-#3);
\fi
\ifnum1=#1
\draw (0,.75-#3) to (A);
\draw (A) to (1,.25-#3);
\fi
\ifnum0=#1
\draw (1,.75-#3) to (A);
\draw (A) to (0,.25-#3);
\fi
}

\newcommand{\opX}[2]{
\ifnum0=#1
\draw(0,1-#2) to (0,-#2);
\draw (0,.75-#2) to[out=-45,in=90] (1,-#2);
\draw (0,.25-#2) to[out=45,in=-90] (1,1-#2);
\else
\draw(1,1-#2) to (1,-#2);
\draw (1,.75-#2) to[out=-135,in=90] (0,-#2);
\draw (1,.25-#2) to[out=135,in=-90] (0,1-#2);
\fi
\draw(2,1-#2) to (2,-#2);
}

\newcommand{\varopT}[3]{
\draw(0,1-#3) to (0,-#3);
\draw(1,1-#3) to (1,-#3);
\coordinate (a) at (.5,.5-#3);
\ifnum1=#2 
\node[antipode, at=(a)] (A) {$S$};
\else
\coordinate (A) at (.5,.5-#3);
\fi
\ifnum1=#1
\draw (0,.75-#3) to (A);
\draw (A) to (1,.25-#3);
\fi
\ifnum0=#1
\draw (1,.75-#3) to (A);
\draw (A) to (0,.25-#3);
\fi
}
\newcommand{\twoantip}[3]{
\coordinate (a) at (0,-#3+.5);
\coordinate (b) at (1,-#3+.5);
\ifnum1=#1
\node[antipode, at =(a)] (A) {$S$};
\else
\coordinate (A) at (0,-#3+.5);
\fi
\ifnum1=#2
\node[antipode, at =(b)] (B) {$S$};
\else
\coordinate (B) at (1,-#3+.5);
\fi
\draw (0,1-#3) to (A);
\draw (A) to (0,-#3);
\draw (1,1-#3) to (B);
\draw (B) to (1,-#3);
}
\newcommand{\tp}[2]{
\begin{minipage}{#1}
\resizebox{#1}{!}{
\begin{tikzpicture}[thick]
#2
\end{tikzpicture}
}
\end{minipage}
}

\newtheorem{proposition}{Proposition}[section]
\newtheorem{theorem}[proposition]{Theorem}
\newtheorem{lemma}[proposition]{Lemma}

\newtheorem{corollary}[proposition]{Corollary}
\theoremstyle{remark}

\theoremstyle{definition}
\newtheorem{definition}[proposition]{Definition}
\newtheorem{remark}[proposition]{Remark}
\newtheorem{conjecture}[proposition]{Conjecture}

\newtheoremstyle{red}{3pt}{3pt}{\color{red}}{}{\itshape}{.}{.5em}{}
\theoremstyle{red}

\makeatletter
\def\Ddots{\mathinner{\mkern1mu\raise\p@
\vbox{\kern7\p@\hbox{.}}\mkern2mu
\raise4\p@\hbox{.}\mkern2mu\raise7\p@\hbox{.}\mkern1mu}}
\makeatother

\tikzset{
empty/.style={inner sep=0pt,minimum size=0mm},
triangle/.style={regular polygon,regular polygon sides =3, draw=red!50,fill=red!20,thick,inner sep=0pt, minimum size=8mm},
semi/.style={semicircle, draw=blue!50,fill=blue!20,thick,inner sep=0pt, minimum size=4mm},
rec/.style={rectangle, draw=green!50,fill=green!20,thick,inner sep=0pt, minimum size=6mm},
operad/.style={circle,draw=red!50,fill=red!20,thick, inner sep=0pt,minimum size=6mm},
hopf/.style={signal, signal to=east, signal from=west,draw=brown!50,fill=brown!20,thick, inner sep=2pt,minimum size=6mm},
lhopf/.style={signal, signal to= west, signal from= east,draw=brown!50,fill=brown!20,thick, inner sep=2pt,minimum size=6mm},
antipode/.style={circle,draw=purple!50,fill=purple!20,thick, inner sep=0pt,minimum size=4mm},
unit/.style={circle,draw=black,fill=white,thick, inner sep=0pt,minimum size=2mm},
break/.style={inner sep=0pt,minimum size=5mm},
block/.style={draw=blue,fill=blue!20,thick,inner sep=10pt}
outer/.style={}
}

\title[The Lie Lie algebra]{The $\Lie$ Lie algebra}

\author[J. Conant]{Jim Conant}

\begin{document}

\maketitle

\begin{abstract}
We study the abelianization of Kontsevich's Lie algebra associated with the Lie operad and some related problems.
Calculating the abelianization is a long-standing unsolved problem, which is important in at least two different contexts: constructing cohomology classes in $H^k(\Out(F_r);\mathbb Q)$ and related groups as well as studying the higher order Johnson homomorphism of surfaces with boundary.
The abelianization carries a grading by ``rank," with previous work of Morita and Conant-Kassabov-Vogtmann computing it up to rank $2$.
This paper presents a partial computation of the rank $3$ part of the abelianization, finding lots of irreducible $\SP$-representations with multiplicities given by spaces of modular forms.
 Existing conjectures in the literature on the twisted homology of $\SL_3(\Z)$ imply that this gives a full account of the rank $3$ part of the abelianization in even degrees. 
\end{abstract}

%\begin{classification}
%17B40; 20C15, 20F28
%\end{classification}

%\begin{keywords}
%Lie algebra of symplectic derivations, automorphism groups of free groups, Johnson cokernel
%\end{keywords}

\section{Introduction}
Let $\mathfrak h$ denote Kontsevich's Lie algebra of symplectic derivations of the free Lie algebra \cite{K1,K2}. 
We call $\mathfrak h$ the ``$\Lie$ Lie algebra" to distinguish it from the Lie algebras of symplectic derivations in the commutative and associative cases. 
 Let $\mathsf A=\oplus \mathsf A_d$ denote the abelianization of the positive degree part $\mathfrak h^+$ graded by degree, let $s_{w}$ be the dimension of the vector space of cusp forms of weight $w$ for the modular group $\PSL_2(\Z)$, and let $[\lambda]_{\SP}$ be the irreducible $\SP$-representation corresponding to the partition $\lambda$. Our main theorem is
\begin{theorem}\label{thm:rank3intro}
Let $\lambda=[a,b,c]$ be a partition of $2m$. The multiplicity of the representation $[\lambda]_{\SP}$ in $\mathsf A_{2m+4}$ is at least $s_{a-b+2}+s_{b-c+2}+\delta_{a,b,c}+\epsilon_{a,b,c}$, where $\epsilon_{a,b,c}=1$ if $a>b>c$ are all even and is equal to $0$ otherwise, and $\delta_{a,b,c}=s_{a-b+2}$ if $a-b=b-c$, and is equal to $0$ otherwise.
\end{theorem}

The  Lie algebra $\mathfrak h$ appears in many places in topology. Kontsevich showed its homology is essentially the same as the direct sum of cohomologies $H^*(\Out(F_r);\F)$ for a base field of characteristic $0$. On the other hand Morita \cite{morita} demonstrated it as the natural target of the higher Johnson homomorphisms of the mapping class group. Relatedly, Levine \cite{Levine} showed that it is the associated graded Lie algebra for the Johnson filtration of the group of 3-dimensional homology cylinders up to homology cobordism.

Calculating the abelianization $\mathsf A$ of the positive degree part  $\mathfrak h^+$ is an important unsolved problem. For example, elements of the abelianization can be assembled to give cohomology classes in $\Out(F_r)$. See \cite{K1,K2,morita,MorClasses,Stable,Ohashi,Gray,CKV1,CKV2} for papers about this. 
The abelianization is bigraded. It has a degree coming from the usual grading on the free Lie algebra $\mathsf L(V)$, but it also carries a grading by ``rank" \cite{CKV1}. We denote the rank $r$ degree $d$ part of the abelianization by $\mathsf A^r_d$, which we think of as a functor on symplectic vector spaces $V$. Theorem~\ref{thm:rank3intro} is a computation in rank $3$, giving representations $[a,b,c]_{\SP}$ in $\mathsf A^3_{2m+4}$. Conjecturally, the inequality of Theorem~\ref{thm:rank3intro} is an equality in even degrees $2m+4$. Indeed this would follow from two conjectures in the literature. See Conjecture~\ref{conj:ash-pollack} of this paper and the remarks just after it. 

The abelianization was previously known up to rank $2$. The ranks 0 and 1 pieces were calculated by Morita using his trace map. The only nonzero pieces are $A^0_1=[1^3]_{\SP}$ and
$A^1_{2k+1}=[2k+1]_{\SP}$ for $k\geq 0$.

The rank $2$ piece was calculated in \cite{CKV1,CKV2}. It is only nonzero for even $d$, in which case the only $\SP$-representations appearing  in $ \mathsf A^{2}_{d+2}$ are $[a,b]_{\SP}$ where $a\geq b+2$, $a+b=d$. These appear with multiplicity $s_{a-b}$ if $a,b$ are even and $s_{a-b}+1$ if $a,b$ are odd.
%cons
%\displaystyle \bigoplus_{a+b=d} [a,b]_{\SP}\otimes \mathcal X_{a-b,a}$, 
%%$$
%%\mathsf A^{2}_{d+2}=
%%\begin{cases}
%%\displaystyle \bigoplus_{a+b=d} [a,b]_{\SP}\otimes \mathcal X_{a-b,a}&\text{if $d$ is even}\\
%%0&\text{if $d$ is odd}
%%\end{cases}
%%$$
%where $\mathcal X_{w,i}$ is equal to $\mathcal S_w$, the vector space of weight $w$ cusp forms if $i$ is even and is equal to $\mathcal M_w$, the vector space of all classical modular forms of weight $w$ if $i$ is odd. 

Unlike in rank $2$, the odd degree part $A^3_{2k+1}$ is not trivial, but it remains poorly understood.  Indeed, computer calculations show that  
\begin{align*}
\mathsf A^{3}_7&\cong [21]_{\SP}\\
\mathsf A^{3}_9&\cong [41]_{\SP}\oplus[32]_{\SP}\\
\mathsf A^{3}_{11}&\cong [7]_{\SP}\oplus2[61]_{\SP}\oplus2[52]_{\SP}\oplus [51^2]_{\SP}\oplus[43]_{\SP}.
\end{align*}
A conceptual explanation of the appearance of these odd degree classes is an important problem for future research.

In addition to the abelianization $\mathsf A$, which is the quotient of $\mathfrak h^+$ by the space of commutators, we consider $\mathsf C$, the quotient of $\mathfrak h^+$ by the Lie algebra generated by degree $1$ elements. $\mathsf C$ inherits the degree from $\mathfrak h$ and  clearly there is a surjection $\mathsf C_d\twoheadrightarrow \mathsf A_d,\,\, d\geq 2.$

By a theorem of Hain, $\mathsf C(H_1(\Sigma_{g,1};\F))$ is isomorphic to the cokernel of the Johnson homomorphism, for large enough genus $g$ compared to degree \cite{Hain}. See \cite{morita, MoritaProspect,EE,C,ES,MSS} for background and historical interest in $\mathsf C$.

Theorem~\ref{thm:rank3intro} gives new information about the Johnson cokernel, but more can be said:
\begin{theorem}\label{thm:omega3intro}
Let $\lambda=[a,b,c]$ be a partition of $2m$. The representation $[\lambda]_{\SP}$ appears in $\mathsf C_{2m+4}$ with multiplicity at least $s_{a-b+2}+s'_{b-c+2}+\delta_{a,b,c}+\epsilon_{a,b,c}$, where
$s'_{2m+2}=\lceil\frac{2m}{3}\rceil-1$ if $m>0$ and $s'_{k}=0$ in all other cases.
\end{theorem} 
Comparing to Theorem~\ref{thm:rank3intro}, the term $s_{b-c+2}\approx (b-c)/12$ for $b-c$ even has been enlarged to $s'_{b-c+2}\approx (b-c)/3$.

Theorem~\ref{thm:omega3intro} is in some sense a rank $3$ computation of classes in the cokernel. In \cite{CK}, we did some rank $2$ computations, which we generalize here. The result is somewhat technical to state, and we  will need a few preliminary definitions. Let $\SF{\lambda}(V)$ be the $\GL(V)$ Schur functor for partition $\lambda$. (When $V$ is suppressed, this is often written $[\lambda]_{\GL}.$) It can be constructed as follows. If $\lambda$ is a partition of $n$, let $P_\lambda$ be the corresponding irreducible representation of the symmetric group $\Sigma_n$. Define $\SF{\lambda}(V)=P_\lambda\otimes_{\F[\Sigma_n]} V^{\otimes n}$, where $\Sigma_n$ acts on $V^{\otimes n}$ by permuting the tensor factors. Next, let $\mathsf L_{(2)}(V)=V\oplus\ext^2 V$ be the free nilpotent Lie algebra on $V$ of nilpotency class $2$, and consider $\SF{\lambda}(\mathsf L_{(2)}(V))$. For any Lie algebra $\mathfrak g$, the adjoint action of $\mathfrak g$ induces an action of $\mathfrak g$ on $\SF{\lambda}(\mathfrak g)$. Modding out by the image of this action yields a vector space $\overline{\SF{\lambda}(\mathfrak g)}$.
Also define $\varphi^{\SP}_{\GL}(\oplus m_\mu [\mu]_{\GL})=\oplus m_\mu[\mu]_{\SP}$, an operator which takes a decomposition into irreducible $\GL$-modules and spits out the corresponding direct sum of $\SP$-modules. 
\begin{theorem}\label{thm:omega2intro}
For each $m>0$, there is an embedding
$$
 \bigoplus_{a+b=2m}\varphi^{\SP}_{\GL}(\overline{\SF{(a,b)}(\mathsf L_{(2)}(V))})^{\otimes \alpha(a,b)}
 \hookrightarrow \mathsf C_{2m+2},$$
 where $\alpha(a,b)=0$ if $a-b$ is odd, $\alpha(a,b)=\lceil\frac{a-b}{3}\rceil$ if $a-b$ is even and $b>0$,  and $\alpha(2m,0)=\lceil\frac{2m}{3}\rceil-1$.
\end{theorem}
In \cite{CK}, roughly speaking, we had proven the same fact with $\alpha(a,b)$ replaced by the smaller number $s_{a-b+2}$.

\subsection{Generalized trace maps}
In order to prove these theorems and related results, 
we will study two functors of cocommutative Hopf algebras: $\homfunctor_r(H)$ and $\Omega_r(H)$ which will have implications for $\mathsf A$ and $\mathsf C$. More precisely $\homfunctor_r(\sym(V))$ will contain complete information about $\mathsf A$, while $\homfunctor_r(T(V))$ and $\Omega_r(T(V))$ will contain a lot of information about $\mathsf C$, including that coming from the Morita and Enomoto-Satoh trace maps \cite{MoritaProspect,ES}.

Supposing that $H$ is a cocommutative Hopf algebra, and that $\mathcal O$ is a cyclic operad, one can form a new cyclic operad $H\mathcal O$, which consists of formal compositions of elements of $H$ and elements of $\mathcal O$, such that one can push elements of $H$ past elements of $\mathcal O$ using the coproduct. (See Definition~\ref{def:freeoperad}). Then $\mathcal G_{H\Lie}$ is defined as the graph complex where vertices are labeled by elements of $H\Lie$. The grading is by number of vertices, so that the bottom homology is $H_1(\mathcal G_{H\Lie})$. Restricting to rank $r$ connected graphs $\mathcal G^{(r)}_{H\Lie}$ we have
$$\homfunctor_r(H):=H_1(\mathcal G^{(r)}_{H\Lie}).
$$
An example of a graph in $\mathcal G^{(3)}_{H\Lie}$ is:
$$\begin{minipage}{5.5cm}
\resizebox{5.5cm}{!}
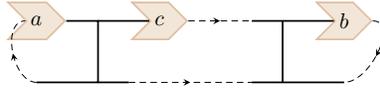
{
\begin{tikzpicture}
\coordinate (a) at (0,0);
\coordinate (b) at(1,0);
\coordinate (c) at (2,0);
\coordinate (d) at (3.5,0);
\coordinate(e) at (4,0);
\coordinate (f) at (5,0);
\coordinate (g) at (5,-1);
\coordinate (h) at (4,-1);
\coordinate (i) at (3.5,-1);
\coordinate (j) at (1.5,-1);
\coordinate(k) at (1,-1);
\coordinate(l) at (0,-1);
\node[hopf, at=(a)](A){$a$};
\node[hopf, at=(c)](C){$c$};
\node[hopf, at=(f)](F){$b$};
\draw[thick] (A) to (b) to (k) to (l);
\draw[thick] (b) to (C);
\draw[thick] (k) to (j);
\draw[thick] (d) to (F);
\draw[thick] (i) to (g);
\draw[thick](e) to (h);
  \begin{scope}[decoration={markings,mark = at position 0.5 with {\arrow{stealth}}}]
\draw[densely dashed, postaction=decorate] (l) to[densely dashed, postaction=decorate, out=180, in=180](A);
\draw[densely dashed, postaction=decorate] (C) to[densely dashed, postaction=decorate,out=0,in=180](d);
\draw[densely dashed, postaction=decorate] (j) to[densely dashed, postaction=decorate,out=0,in=180](i);
\draw[densely dashed, postaction=decorate] (F) to[densely dashed, postaction=decorate,out=0,in=0](g);
\end{scope}
\end{tikzpicture}
}
\end{minipage}
$$
Here $a,b,c$ represent elements in $H$, while there are two solid trees representing elements of $\Lie((4))$.

In \cite{CK}, we proved that $\homfunctor_r(H)\cong H^{2r-3}(\Out(F_r);\overline{H^{\otimes r}})$, where $\Aut(F_r)$ acts on $H^{\otimes r}$ via the Hopf algebra structure (see section~\ref{sec:action}), and $\overline{H^{\otimes r}}$ is a natural quotient on which inner automorphisms act trivially. When $H=\sym(V)$, $\overline{H^{\otimes r}}=H^{\otimes r}$ and the action of $\Out(F_r)$ factors through the standard $\GL_r(\Z)$ action.

By definition $\homfunctor_r(H):=\mathcal G^{(r)}_{H\Lie,1}/\partial(\mathcal G^{(r)}_{H\Lie,2})$.  We define $\Omega_r(H)$ to be $\mathcal G^{(r)}_{H\Lie,1}/\partial\mathcal S_2$ where $\mathcal S_2$ consists of 2-vertex graphs where one of the two vertices is labeled either by a tripod (generator of $\Lie((3))$) or an element of $H$.

In \cite{CKV1} we introduced a generalized \emph{trace} map $\Tr\colon \mathfrak h\to \mathcal G_{{\sym(V)\Lie},1}$. The Lie algebra $\mathfrak h$ has a well-known description via trees with $V$-labeled leaves, and the trace map is defined by adding several directed edges joining pairs of leaves of a disjoint union of trees in all possible ways, multiplying by the product of contractions of the labels. The leaves which are not joined become hairs, and strings of adjacent $V$-labeled hairs can be interpreted as elements of $\sym(V)$ \cite{CK}. 

In fact, in \cite{CKV2}  the trace map was used to prove that $${\mathsf A}^r\cong \varphi^{\SP}_{\GL}(\homfunctor_r(\sym(V))).$$
Thus $\homfunctor_r(\sym(V))$ contains complete information about the abelianization $\mathsf A$. 

The map $\Tr$ also induces maps $\mathsf C\to \Omega_r(T(V))$ and $\mathsf C\to \homfunctor_r(T(V))$ for the tensor algebra $T(V)$ \cite{C,CK}, and moreover $\mathsf C$ surjects onto $\varphi^{\SP}_{\GL}(\Omega_r(T(V)))$ and $\varphi^{\SP}_{\GL}(\homfunctor_r(T(V)))$.

The meat of the paper is to find explicit presentations for $\homfunctor_r(H)$ and $\Omega_r(H)$ for $r\leq 3$ (Theorems~\ref{thm:rank2pres}, \ref{thm:omega2pres}, \ref{thm:rank3pres}, \ref{thm:omega2pres}). In particular, when $H=\sym(V)$, the presentation for $\homfunctor_3(H)$ implies that 
\begin{theorem}
$$\homfunctor_3(\sym(V))_{2k}\cong H^3(\Out (F_3);\sym(V)^{\otimes 3})_{2k}\cong H^3(\GL_3(\Z);\sym(V)^{\otimes 3})_{2k}$$
\end{theorem}
\noindent allowing us to appeal to known results about the cohomology of $\GL_3(\Z)$ to produce classes in $\mathsf A^3$ and leading to Theorem~\ref{thm:rank3intro}. Note that the equality of $\Out(F_3)$ and $\GL_3(\Z)$ cohomology in even degrees here is a novel and unexpected result. Its proof comes down to finding presentations for both as vector spaces and noticing they are equal. Surely a more conceptual proof exists, which is another good problem for future research. It is worthwhile to note that the isomorphism does not hold in odd degrees, as the $\GL_3(\mathbb Z)$ cohomology vanishes but the $\Out(F_3)$ cohomology does not.

 Theorem~\ref{thm:omega3intro} is proven by careful examination of  calculations of Allison-Ash-Conrad \cite{allison-ash-conrad} for $H_3(\GL_3(\Z),M)$, extending them to our presentation of $\Omega_3(\sym(V))$.

Similarly, the presentation of $\Omega_2(\sym(V))$ allows us to calculate its dimension using an elementary and clever argument due to Martin Kassabov, and extending the ideas of \cite{CK}, one can compute $\Omega_2(U\mathsf L_{(2)})$, which is a quotient of $\Omega_2(T(V))$, leading to Theorem~\ref{thm:omega2intro}.

These presentations are also useful for computer calculations, the results of which are listed in the last section. These computations have been confirmed and extended in \cite{CKS}.

\subsection{The groups $\Gamma_{n,s}$}
The abelianization $\mathsf A$ is intimately related to the cohomology of certain groups $\Gamma_{n,s}$ in their vcd, and our computations of the rank $3$ part of the abelianization have implications for the cohomology of $\Gamma_{3,s}$. First we recall the definition of the groups $\Gamma_{n,s}$. Let $X_{n,s}$ be a 1-complex homotopy equivalent to a wedge of $n$ circles with $s$ marked points. $\Gamma_{n,s}$ is the set of self-homotopy equivalences of $X_{n,s}$ fixing the $s$ points, up to homotopy relative to those points. These are groups which generalize $\Out(F_n)$ and $\Aut(F_n)$ in the sense that $\Gamma_{n,0}=\Out(F_n)$ and $\Gamma_{n,1}=\Aut(F_n)$. 
The vcd of $\Gamma_{n,s}$ is $2n-3+s$. 

Let $V^{\wedge n}$ be the $\Sigma_n$-representation which is $V^{\otimes n}$ as a vector space, with $\Sigma_n$ acting with the sign of the permutation.
We have the following theorem \cite{CKV1,CHKV}
\begin{theorem}
There is an isomorphism
$$H^{2n-3}(\Out(F_n);\sym(V)^{\otimes n})\cong \bigoplus_{s\geq 0} H^{2n-3+s}(\Gamma_{n,s};\F)\otimes V^{\wedge s}.$$
\end{theorem}

This implies that the $\GL$ representations appearing in $H^{2n-3}(\Out(F_n);\sym(V)^{\otimes n})$ appear as the conjugate Young diagram for an $\Sigma_s$-representation in $H^{2n-3+s}(\Gamma_{n,s};\F)$. Thus Theorem~\ref{thm:rank3intro} implies the following theorem.

\begin{theorem}
Let $\lambda=[a,b,c]$ be a partition of $2s$, and $\lambda^*$ its conjugate. Then the $\Sigma_{2s}$-representation $P_{\lambda^*}$ appears in $H^{3+s}(\Gamma_{3,2s};\F)$ with multiplicity at least $s_{a-b+2}+s_{b-c+2}+\delta_{a,b,c}+\epsilon_{a,b,c}$.
\end{theorem}

{\bf Acknowledgments:} The author thanks the referee, Avner Ash, Martin Kassabov, Andy Putman and Nolan Wallach for useful discussions and suggestions.
\section{Definitions}
\subsection{The $\Lie$ Lie algebra}
Suppose $V$ is a finite dimensional symplectic $\F$-vector space. I.e. it has a nondegenerate antisymmetric bilinear form  $\langle\cdot,\cdot\rangle\colon V\otimes V\to \F$. Let $\{p_i,q_i\}$ form a symplectic basis:
$\langle p_i,q_i\rangle =1=-\langle q_i,p_i\rangle$ and all other pairings of basis elements are $0$.
Consider the free Lie algebra $\mathsf{L}(V)$ which has $V$ as its degree $1$ elements. Define $\mathfrak h(V)=\operatorname{Der}_\omega(\mathsf L(V))$ to be the set of derivations of $\mathsf L(V)$ which annihilate the element $\omega=\sum [p_i,q_i]$, and let $\mathfrak h^+(V)$ be generated by derivations of positive degree. $\mathfrak h(V)$ is well-known to be isomorphic to the space of Lie spiders. These are trivalent trees with univalent vertices labeled by elements of $V$, modulo orientation, IHX and multilinearity relations. See \cite{CV, Levine} for more details. Let $\displaystyle\mathfrak h_\infty^+=\lim_{n\to\infty}\mathfrak h^+(V_n)$ where $$\cdots \to V_n\to V_{n+1}\to\cdots$$ is a standard sequence of symplectic vector spaces $V_n$ of dimension $2n$.

In this paper we are primarily interested in this Lie algebra as it relates to the homology of $\Out(F_n)$ and the Johnson homomorphism of the mapping class group of a surface. 
\begin{theorem}[Kontsevich]
$\displaystyle \lim_{n\to\infty}PH^*(\mathfrak h^+(V_n))^{\SP}\cong \bigoplus_{r\geq 2} H^*(\Out(F_r);\F)$
\end{theorem}
In particular the abelianization gives rise to potential homology classes in $\Out(F_r)$:
$$\lim_{n\to\infty} PH^*(\mathsf A^+(V_n))^{\SP}\to  \bigoplus_{r\geq 2} H^*(\Out(F_r);\F)$$
On the other hand, the higher order Johnson homomorphism is a Lie algebra homomorphism $$\tau\colon \mathsf{Gr}_{\mathbb J}(\Mod(g,1)) \to \mathfrak h^+(H_1(\Sigma_{g,1};\F))$$
where $ \mathsf{Gr}_{\mathbb J}(\Mod(g,1))$ is the associate graded (tensored with $\F$) vector space associated with the Johnson filtration of the mapping class group $\Mod(g,1)$.
\begin{theorem}[Hain]
$\im\tau$ is (stably) generated as a Lie algebra by degree $1$ elements.
\end{theorem}
Let $\mathsf C(V)$ be $\mathfrak{h}^+(V)$ divided by the Lie algebra generated by degree $1$ elements. By Hain's result, $\mathsf C_d(H_1(\Sigma_{g,1};\F))$ this is isomorphic to the degree $d$ part of the Johnson cokernel when the $d$ is small compared to the genus $g$. Clearly, for degree $d>1$, we have a surjection $\mathsf C_d(V)\twoheadrightarrow \mathsf A_d(V)$.

\subsection{The operad $H\cO$}
Suppose $H$ is a co-commutative Hopf algebra and $\cO$ is an operad with unit (in the category of $\F$-vector spaces). We let $\cO[n]$ denote the vector space spanned by operad elements with $n$ inputs and one output, $n$ being referred to as the \emph{arity}. If $\cO$ is cyclic, we let $\cO\arity{n}=\cO[n-1]$ as a $\Sigma_{n}$-module.

Regard $H$ as an operad with elements only of arity $1$ and operad composition given by algebra multiplication. The antipode $S$ turns $H$ into a cyclic operad: the $\Sigma_{2}$ action sends $h$ to $S(h)$.

\begin{definition}\label{def:freeoperad}\
\begin{enumerate}
\item Let $\cO_1$ and $\cO_2$ be operads with unit. Define $\cO_1*\cO_2$ to be the operad \emph{freely generated by $\cO_1$ and $\cO_2$}. This is defined to be the operad consisting of trees with vertices of valence $\geq 2$ labeled by elements of $\cO_1$ or $\cO_2$. Composing two elements of $\cO_i$ for $i=1,2$ along a tree edge is considered the same element of $\cO_1*\cO_2$, and the units of $\cO_1$ and $\cO_2$ are identified and equal to the unit of $\cO_1*\cO_2$.
\item Let $H\cO$ be the quotient of $H*\cO$ by the relation that $h$ commutes with an element of $\cO$ via the comultiplication map as in the figure below. (The fact that $1_\cO=1_H$ is also included for emphasis.)
\end{enumerate}
\end{definition}

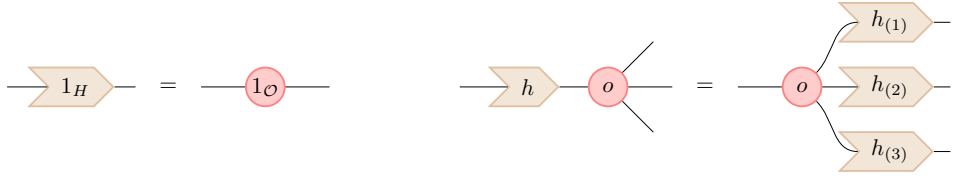
\begin{center}
\begin{minipage}{\linewidth}
\resizebox{\linewidth}{!}{
\begin{tikzpicture}
\node[empty](aa){};
\node[operad](bb)[left of=aa]{$1_\cO$} edge (aa);
\node[empty](cc)[left of=bb]{}  edge (bb);

\node[empty](dd)[left of=cc]{};
\node[hopf](ee)[left of=dd]{$\,\,1_H$}  edge (dd);
\node[empty](ff)[left of=ee]{}  edge (ee);

\path(cc) to node{$=$} (dd);

\node[break](br)[right of=aa]{};

\node[empty](a)[right of=br]{};
\node[hopf](b)[right of=a]{$\,\,h$} edge (a);
\node[operad](c)[right of=b, node distance=1.3cm]{$o$} edge (b);
\node[empty](d)[right of=c]{} edge (c);
\node[empty](e)[above right of=c]{} edge (c);
\node[empty](f)[below right of=c]{} edge (c);

\node[empty](a')[right of=d]{};
\node[operad](b')[right of=a']{$o$} edge (a');
\node[hopf](c')[right of=b', node distance=1.3cm]{$\,\,h_{(2)}$} edge (b');
\node[empty](d')[right of=c']{} edge (c');
\node[hopf](e')[above of=c']{$\,\,h_{(1)}$}; \draw (b') to[out=45,in=180] (e');
\node[empty](f')[right of=e']{} edge (e');
\node[hopf](g')[below of=c']{$\,\,h_{(3)}$}; \draw (b') to[out=-45,in=180] (g');
\node[empty](h')[right of=g']{} edge (g');

\path(d) to node{$=$} (a');

\end{tikzpicture}}
\end{minipage}
\end{center}

Note the use of Sweedler notation hiding the fact that the coproduct is actually a sum of pure tensors. 

\subsection{$\Aut(F_n)$ acting on $H^{\otimes n}$}\label{sec:action}
In \cite{CK}, a right action of $\Aut(F_n)$ on $H^{\otimes n}$ is defined. It can be described as follows. For a group $G$, introduce maps $m\colon G\times G\to G$, $\Delta\colon G\to G\times G$ and $S\colon G\to G$, defined by $m(g,h)=gh$, $m(g)=g\times g$ and $S(g)=g^{-1}$. Suppose $F_n$ is generated by $x_1,\ldots x_n$. Given $\phi\in \Aut(F_n)$ represent the transformation $F_n^n\to F_n^n$ given by $(x_1,\ldots,x_n)\mapsto (\phi(x_1),\dots,\phi(x_n))$ as a composition of permutations of coordinates, multiplication of two coordinates, doubling of a coordinate and inversion of a coordinate. Now think of these operations instead as operations of the Hopf algebra. For example suppose $\phi(x_1)=x^2_2x_1^{-1}$ and $\phi(x_2)=x_1x_2^{-1}$. Then $(a\otimes b)\cdot \phi=b_{(1)}b_{(2)}S(a_{(1)})\otimes a_{(2)}S(b_{(3)})$.

We let $\rho_H\colon \Aut(F_n)\to \Aut(H^{\otimes n})$ denote this action.
\begin{definition}
The Hopf algebra $H$ acts on $H^{\otimes n}$ via \emph{conjugation}. That is, suppose $h\in H$ and $\Delta^{2n}(h)=h_{(1)}\otimes h_{(2)}\otimes\cdots\otimes h_{(2n-1)}\otimes h_{(2n)}$, using Sweedler notation. Then define
$$
h\circledast (h_1\otimes\cdots\otimes h_n)= h_{(1)}h_1S(h_{(2)})\otimes\cdots\otimes h_{(2n-1)}h_n S(h_{(2n)}).
$$
Let $\overline{H^{\otimes n}}$ be the quotient of $H^{\otimes n}$ by the subspace spanned by elements of the form
$$
(h-\eta\epsilon(h))\circledast(h_1\otimes \cdots\otimes h_n),
$$
i.e., this is the maximal quotient of $H^{\otimes n}$ where the conjugation action of $H$ factors through the counit.

\end{definition}

The action of $\Aut(F_n)$ induces an action of $\Out(F_n)$ on $\overline{H^{\otimes n}}$, which we also denote by $\rho_H$.

\subsection{Graph homology}
Recall from~\cite{CV} that one can define a graph complex $\cG_\cO$ for any cyclic operad $\cO$ by putting elements of $\cO\arity{|v|}$ at each vertex $v$ of a graph and identifying the i/o slots with the adjacent edges. In this definition, the graph may have bivalent vertices but no univalent or isolated vertices, since the operad $\cO$ is assumed not to have anything in arity $-1$ and $0$.

These complexes are graded by the number of vertices of the underlying graph and the boundary operator is induced by contracting edges of the underlying graph.  Let $\cG_{\cO}^{(n)}$ 
be the subcomplex spanned by $\cO$-colored connected graphs of rank $n$. 
In this section we will study $H_\bullet(\cG_{H\Lie})$.
As a consequence of the definitions, the rank $0$ part of $\cG_{H\Lie}$ 
is trivial.

\begin{definition}
Let $\overline{\cG_{H\Lie}}$ denote the quotient of the graph complex for $H\Lie$ where the elements in $H$ are allowed to slide through the edges, i.e., the following graphs in $\cG_{H\Lie}$ are equivalent in $\overline{\cG_{H\Lie}}$.
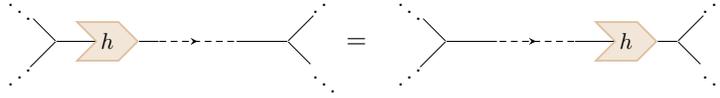
\begin{center}
\begin{minipage}{4.5cm}
\resizebox{4.5cm}{!}{
\begin{tikzpicture}
\node[hopf](a){$h$};
\node[empty](b)[left of=a, node distance=.8cm]{} edge (a);
\node[empty](c)[above left of=b, node distance=.8cm]{$\ddots$} edge (b);
\node[empty](d)[below left of=b, node distance=.8cm]{$\Ddots$} edge (b);
\node[empty](e)[right of=a, node distance=.8cm]{} edge (a);
\node[empty](f)[right of=e, node distance=1.2cm]{};
\begin{scope}[decoration={markings,mark = at position 0.5 with {\arrow{stealth}}}]
\draw[decorate] (e) to (f);
\end{scope}
\draw[densely dashed] (e) to (f);
\node[empty](g)[right of=f, node distance=.8cm]{} edge (f);
\node[empty](h)[above right of =g, node distance=.8cm]{$\Ddots$} edge (g);
\node[empty](i)[below right of =g, node distance=.8cm]{$\ddots$} edge (g);
\end{tikzpicture}}
\end{minipage}
$=$\,\,\,
\begin{minipage}{4.5cm}
\resizebox{4.5cm}{!}{
\begin{tikzpicture}
\node[hopf](a){$h$};
\node[empty](b)[right of=a, node distance=.8cm]{} edge (a);
\node[empty](c)[above right of=b, node distance=.8cm]{$\Ddots$} edge (b);
\node[empty](d)[below right of=b, node distance=.8cm]{$\ddots$} edge (b);
\node[empty](e)[left of=a, node distance=.8cm]{} edge (a);
\node[empty](f)[left of=e, node distance=1.2cm]{};
\begin{scope}[decoration={markings,mark = at position 0.5 with {\arrow{stealth}}}]
\draw[decorate] (f) to (e);
\end{scope}
\draw[densely dashed] (e) to (f);
\node[empty](g)[left of=f, node distance=.8cm]{} edge (f);
\node[empty](h)[above left of =g, node distance=.8cm]{$\ddots$} edge (g);
\node[empty](i)[below left of =g, node distance=.8cm]{$\Ddots$} edge (g);
\end{tikzpicture}}
\end{minipage}
\end{center}
It is clear that the quotient map $\cG_{H\Lie} \to \overline{\cG_{H\Lie}}$ preserves the differential and induces a map between the homologies.   
\end{definition}

\begin{remark}
The Hopf algebra elements $h\in H$ encode the ``hairs" of hairy graph homology. For example the product $v_1v_2\cdots v_k\in T(V)$ represents $k$ hairs in a row, labeled by $v_1,\ldots, v_k$. When $H=\sym(V)$, these hairs commute, which is what happens for hairy Lie graph homology.  See section 5.3 of \cite{CKV2}. 
\end{remark}

\begin{theorem}[Conant-Kassabov]
\label{thm:HReduced}
For $n\geq 2$ we have $$H_k(\overline{\cG_{H\Lie}^{(n)}}) = H^{2n-2-k}(\Out(F_n);\overline{H^{\otimes n}}).$$
\end{theorem}
\begin{proof}
See \cite{CK}.
\end{proof}

Let $\homfunctor_n(H)=H_1(\overline{\cG_{H\Lie}^{(n)}})$.

\begin{theorem}\
\begin{enumerate}
\item There is a stable embedding
$$\lim_{n\to\infty}\mathsf A(V_n)\hookrightarrow \lim_{n\to\infty} \ext^3 V_n\oplus \bigoplus_{r\geq 1}\homfunctor_r(\sym(V_n)).$$
 Moreover the $\SP$-decomposition of $\mathsf A$ is isomorphic to the $\GL$ decomposition of $\bigoplus_{r\geq 1}\homfunctor_r(\sym(V))$: i.e. $\varphi^{\SP}_{\GL}(\bigoplus_{r\geq 1}\homfunctor_r(\sym(V)))=\mathsf A^+$.
\item
Then there is a map from the Johnson cokernel $$\mathsf C(V)\to\bigoplus_{r\geq 1}\homfunctor_r(T(V)).$$ Stably, every $\GL$-representation in $\homfunctor_r(T(V))$ will appear as an $\SP$ representation in $\mathsf C(V)$. 
\end{enumerate}
\end{theorem}
\begin{proof}
The first statement follows from \cite{CKV1,CKV2}. The second from \cite{CK}.
\end{proof}
\subsection{Johnson cokernel obstructions}
We define $\Omega_n(H)$ to be $\overline{\cG^{(n)}_{H\Lie,1}}$ modulo certain boundaries.
We define a subspace $\mathcal S_2\subset \cG^{(n)}_{H\Lie,2}$ as follows.
Note that graphs in  $\cG^{(n)}_{H\Lie,2}$ are described by two elements of $H\Lie$ joined by some edges. $\mathcal S_2$ is spanned by graphs where one of the two $H\Lie$ elements is actually an element of $\Lie((3))\subset H\Lie((3))$, i.e. it is a tripod where all three i/o slots are joined to graph edges.  %The second type is where one of the elements of $H\Lie$ is actually an element of $V\subset H\subset H\Lie((2))$, i.e. a tripod where one hair has an element of $V$ decorating it, and the two other i/o slots are connected by edges to the the other $H\Lie$ element. 
We define $$\Omega_n(H)=\overline{\cG^{(n)}_{H\Lie,1}}/\partial \mathcal S_2.$$ 

Let $\sf C$ be the cokernel of the Johnson homomorphism.

\begin{theorem} 
There is a degree preserving $\SP$-module map $$\Tr^{C}\colon \mathsf C\to \bigoplus_{n\geq 0} \Omega_n(T(V))$$ which surjects onto the top level $\SP$-representations.
\end{theorem}
\begin{proof}
This is the main theorem of \cite{C}, with slightly altered notation. Our $\Omega_n(T(V))$ was called simply $\Omega_n(V)$ in \cite{C}.
\end{proof}

\begin{corollary}
There is a degree preserving $\SP$-module map $$\Tr^{C}\colon \mathsf C\to \bigoplus_{n\geq 0} \Omega_n(U\mathsf L_{(k)}(V))$$ which surjects onto the top level $\SP$-representations.
\end{corollary}
\begin{proof}
The map $T(V)\twoheadrightarrow U(\mathsf L_{(k)}(V))$ induces a surjection $\Omega_n(T(V))\twoheadrightarrow \Omega_n(U(\mathsf L_{(k)}(V)))$. To see that $\Omega_n$ preserves surjections $H_1\twoheadrightarrow H_2$, observe that there is a map of chain complexes $\overline{\cG^{(n)}_{H_1\Lie,\cdot}}\twoheadrightarrow \overline{\cG^{(n)}_{H_2\Lie,\cdot}}$ which preserves the $\mathcal S_2$ subspace. So the induced map is necessarily a surjection: $\Omega_n(H_1)\twoheadrightarrow \Omega_n(H_2)$.
\end{proof}

\section{Rank 0 and 1}
Rank $0$ is a bit special. $\homfunctor_0(V)$ is not well defined. However, 
  the degree $1$ part of $\mathfrak h^+(V)$ is isomorphic to $\ext^3(V)\cong [1^3]_{\SP}\oplus [1]_{\SP}$.
The $[1]_{\SP}$ is detected in rank $1$ by the trace map, while $[1^3]_{\SP}$ lies in the kernel of the trace. So there is a sense in which $\homfunctor_0(V)=[1^3]_{\SP}.$
   
Rank 1 was considered in \cite{CK}.
\begin{theorem}[Conant-Kassabov]
There is an isomorphism $\homfunctor_1(H)\cong \Omega_1(H)  \cong(\id - S)(H/[H,H])$. In particular
\begin{enumerate}
\item $\displaystyle\homfunctor_1(\sym(V))\cong   \bigoplus_{k\geq 0}\sym^{2k+1}(V)$
\item $\displaystyle\homfunctor_1(T(V) \cong )\bigoplus_{k\geq 1} [V^{\otimes k}]_{D_{2k}}$
\end{enumerate}
\end{theorem}
These give the targets of the Morita and Enomoto-Satoh traces respectively.

\section{Rank 2}
In this section we calculate presentations for $\homfunctor_2(H)$ and $\Omega_2(H)$. $\homfunctor_2(\sym(V))$ was completely calculated in \cite{CKV1} while $\homfunctor_2(T(V))$ was partially analyzed in \cite{CK}. In particular,  $\homfunctor_2(U\mathsf L_{2}(V))$ was calculated in \cite{CK}, where $\mathsf L_{(2)}(V)\cong V\oplus \ext^2V$ is the free Lie algebra of nilpotency class $2$. Our calculations in this section allow us to completely calculate $\Omega_2(\sym(V))$ and $\Omega_2(U\mathsf L_{2}(V))$ as well, leading to new families of representations in the Johnson cokernel.

\subsection{Operators on $\overline{H^{\otimes 2}}$}
Note that $\Out(F_2)\cong \GL_2(\Z)$ acts on $\overline{H^{\otimes 2}}$. We introduce some standard matrices in $\GL_2(\Z)$ and explore how they act on $\overline{H^{\otimes 2}}$.
Let $$s=\begin{bmatrix}
0&1\\-1&0
\end{bmatrix},
t=\begin{bmatrix}
1&1\\0&1
\end{bmatrix},
\text{ and thus }
st=\begin{bmatrix}
0&1\\-1&-1
\end{bmatrix}.
$$
Let $\gamma=\rho_H(st)$. The operator $\gamma$ is induced by the automorphisms $x\mapsto y^{-1}, y\mapsto xy^{-1}$ and $x\mapsto y^{-1}, y\mapsto y^{-1}x$, which are equivalent as outer automorphisms. Hence, one can write $(a\otimes b)\cdot \gamma =S(b_{(1)})\otimes a S(b_{(2)})$ or $(a\otimes b)\cdot \gamma =S(b_{(1)})\otimes S(b_{(2)})a$, and indeed both are equivalent in $\overline{H^{\otimes 2}}$. 
Thus,
$$(a\otimes b)(\id +\gamma+\gamma^2)=a\otimes b+S(a_{(1)}) b\otimes S(a_{(2)})+ S(b_{(1)})\otimes S(b_{(2)})a.$$

Let $\delta=\begin{bmatrix}
-1&0\\0&1
\end{bmatrix}$, and $\tau=\begin{bmatrix}
0&1\\1&0
\end{bmatrix}.$

\subsection{Cohomology of $\GL_2(\Z)$}
In this section, we calculate a presentation for $H^1(\GL_2(\Z);M)$ where $M$ is a right $\GL_2(\Z)$-module. (That is, $M$ is a vector space with a $\GL_2(\Z)$ action.)
This is well-known and intimately related to the Eichler-Shimura isomorphism and modular symbols, but as it is relatively easy, we include a proof.
 Start with the fact that $\PSL_2(\Z)\cong \Z_2*\Z_3$ with the first factor generated by $\smat$ and the second factor generated by $\smat\tmat$. We also use the fact that $H^1(G;M)$ is isomorphic to the set of derivations of $G$ modulo inner derivations. That is functions $\phi\colon G\to M$ satisfying $\phi(gh)=\phi(h)+\phi(g)h$ modulo cocyles of the form $\phi_m(g)=m-mg$ for each $m\in M$. 

We begin with a well-known lemma.
\begin{lemma}\label{lem:order}
Let $X$ be an operator acting on the $\F$ vector space $M$ on the right which satisfies $X^n=\id$. Define $\xi_X=\id+X+\cdots+X^{n-1}$. Then $\ker \xi_X=\im(\id-X)$ and $\ker(\id-X)=\im(\xi_x)$.
\end{lemma}
\begin{proof}
A slick way to prove $\ker \xi_X=\im(\id-X)$ is to observe that $H^1(\mathbb Z_n;M)=0$ since $M$ is over a field of characteristic $0$. This means that all $1$ cocycles are inner derivations. In our case the $1$ cocycles $\phi\colon \mathbb Z_n\to M$ are identified with the image on a generator $\phi(X)\in M$, and because $X^n=\id$, must satisfy $\phi(X)\in\ker\xi_X$. Thus the cocycles are exactly $\ker\xi$. Similarly the coboundaries (or inner derivations) 
are canonically identified with $\im(\id-X)$.

A similar argument uses $H_1(\mathbb Z_n;M)=0$ to show $\ker(\id-X)=\im(\xi_x)$.
\end{proof}
\begin{proposition} 
Let $G=\Z_k*\Z_\ell=\langle X,Y\,|\, X^k=Y^\ell=1\rangle$, and 
let $M$ be a right $G$-module. 
 We have an isomorphism
$H^1(\Z_k*\Z_\ell;M)\cong \dfrac{M}{\langle \xi_X,\xi_Y\rangle},$
where the notation $\langle r_1,r_2,\ldots r_k\rangle$ signifies the subspace $Mr_1+\cdots +Mr_k$.
\end{proposition}
\begin{proof}
%Let $\xi_X= \id+X+\cdots+X^{k-1}$ and $\xi_Y=\id+Y+\cdots Y^{\ell-1}$.
Using the definition of $H^1(G;M)$ as the space of derivations modulo inner derivations, we see
\begin{align*}H^1(\Z_k*\Z_\ell;M)&\cong 
\dfrac{\{(\phi(X),\phi(Y)\,|\,\phi(X)\xi_X=0, \phi(Y)\xi_Y=0\}}
{\{(m-mX,m- mY)\}}\\
&\cong\dfrac{\{(a,b)\,|\,a\xi_X=0, b\xi_Y=0\}}
{\{(m-mX,m-mY)\}}\\
&\cong\dfrac{\{(a(\id-X),b(\id-Y))\,|\, a,b\in M\}}
{\{(m(\id-X),m(\id-Y)\,|\, m\in M)\}}
\end{align*}
Define a map $M\to \dfrac{\{(a(\id-X),b(\id-Y))\,|\, a,b\in M\}}
{\{(m(\id-X),m(\id-Y)\,|\, m\in M)\}}$ by $z\mapsto (0,z(\id-Y))$. It is straightforward to see that the kernel of this map is $\im\xi_X+\im\xi_Y$, completing the proof.
\end{proof}
\begin{corollary}
Let $M$ be a right $\PSL_2(\Z)$ module. Then
$$H^1(\PSL_2(\Z);M)\cong \dfrac{M}{\langle1+s,1+(st)+(st)^2\rangle}.$$
\end{corollary}

\begin{corollary}
Suppose that $M$ is a right $\SL_2(\Z)$-module. If $-I_2$ acts as the identity on $M$, then 
  $M$ is a $\PSL_2(\Z)$-module under the induced action and
there is an isomorphism
$$H^1(\SL_2(\Z);M)\cong H^1(\PSL_2(\Z);M)$$
%\dfrac{M}{\langle 1+\smat\tmat+(\smat\tmat)^2,1+\smat\rangle}.$$
\end{corollary}
\begin{proof}
Using the Hochschild-Serre spectral sequence, $$H^1(\SL_2(\Z);M)\cong H^1(\PSL_2(\Z);H^0(\Z_2;M))\cong H^1(\PSL_2(\Z);M^{\Z_2}).$$ By hypothesis $M^{\Z_2}=M$.
\end{proof}

\begin{proposition}\label{prop:glpres}
Suppose that $M$ is a right $\GL_2(\Z)$-module. We have an isomorphism
$$H^1(\GL_2(\Z);M) \cong  \dfrac{M}{\langle 1+\smat\tmat+(\smat\tmat)^2,(1+\smat),(1-\tau)\rangle}.$$  
\end{proposition}
\begin{proof}
Let $M=M^+\oplus M^-$ be the decomposition into eigenspaces of the action by $-I_2$. Since $-I_2$ is in the center of $\GL_2(\Z)$, these are also $\GL_2(\Z)$ representations.
The Hochschild-Serre spectral sequence yields $$H^1(\GL_2(\Z);M)\cong H^0(\Z_2;H^1(\SL_2(\Z);M))\cong [H^1(\SL_2(\Z);M)]^{\Z_2}
=H^1(\SL_2(\Z);M^+)
$$ where the $\mathbb Z_2$ acts via the matrix $\tau$. Now identify the $\Z_2$ invariants with coinvariants, which will kill $M^-$, to complete the proof.
\end{proof}

\begin{corollary}\label{cor:detrep}
Let $(det)$ be the $1$ dimensional representation of $\GL_2(\Z)$ given by the determinant, and suppose $M $ is a right $\GL_2(\Z)$-module. Then $$H^1(\GL_2(\Z);M\otimes(det)) \cong  \dfrac{M}{\langle 1+\smat\tmat+(\smat\tmat)^2,(1+\smat),(1+\tau)\rangle}.$$ 
\end{corollary}

\begin{remark}[Relation between twisted homology and cohomology]\label{rem:dual}
Suppose $G$ is a group and $M$ is a finite dimensional $\F[G]$-module over a field of characteristic $0$. Then $H^*(G,M)\cong (H_*(G,M^*))^*$ where the $*$ superscript denotes the vector space dual.

In particular, $H_1(\GL_2(\Z);M\otimes(det))$ can be identified with the subspace of $M$ consisting of all $m$ satisfying 
$m(1+\smat\tmat+(\smat\tmat)^2)=m(1+\smat)=m(1+\tau)=0$. 
\end{remark}

\subsection{Cohomology of $\GL_2(\Z)$ with coefficients in irreducible $\GL_2(\F)$ representations} 
For every partition $(a,b)$, $a\geq b$, there is an irreducible $\GL_2(\F)$ representation $\Psi_{(a,b)}$. In terms of Schur functors, it is defined as $\Psi_{(a,b)}=\SF{(a,b)}(\F^2)$. All irreducible representations are of the form $\Psi_{(a,b)}\otimes (\det)^k$. Moreover $\Psi_{(a+k,b+k)}\cong \Psi_{(a,b)}\otimes (\det)^{k}$. Letting $\Phi_g=\Psi_{(g,0)}\cong\sym^g(\F^2)$, it follows that
all representations of $\GL_2(\F)$ are of the form $\Phi_g\otimes(\det)^k$ for some $g$ and $k$.

\begin{proposition}
The cohomology group
$H^1(\GL_2(\Z);\Phi_g\otimes (\det)^k)$ is trivial if $g$ is odd, is equal to $\mathcal S_{g+2}$ if $g,k$ are even and is equal to $\mathcal M_{g+2}$ if $g$ is even but $k$ is odd.
\end{proposition}
\begin{proof}
The Eichler-Shimura isomorphism implies that $H^1(\SL_2(\Z);\Phi_g)\cong\mathcal M_{g+2}\oplus\mathcal S_{g+2}$. Passing to $\GL_2(\Z)$ involves taking $\mathbb Z_2$-invariants of this space. This is done explicitly in \cite{CKV1}, and gives the result stated in the proposition.
\end{proof}

\subsection{Presentations for $\homfunctor_2(H)$ and $\Omega_2(H)$}
We begin with the statements of the theorems. The proofs are found in section~\ref{sec:proofsrank2}.
\begin{theorem}\label{thm:rank2pres}
There is an isomorphism
$$\homfunctor_2(H)\cong \frac{\overline{H^{\otimes 2}}}{\langle\id-\tau,\id+\delta,\id+\gamma+\gamma^2\rangle }.$$
I.e.  $\homfunctor_2(H)$
is the quotient of $\overline{H^{\otimes 2}}$ by 
\begin{enumerate}
\item  $a\otimes b=b\otimes a$%=S(a)\otimes S(b)$
\item  $a\otimes b=-S(a)\otimes b$
\item   $a\otimes b+S(a_{(1)}) b\otimes S(a_{(2)})+ S(b_{(1)})\otimes S(b_{(2)}) a=0.$ 
\end{enumerate}
\end{theorem}
It can be useful for visualization purposes, especially when we move to rank $3$ to represent an element of $H^{\otimes n}$, graphically as a rectangle with $n$ inputs representing the $n$ tensor factors. See Figure ~\ref{fig:rank2}.
\begin{figure}
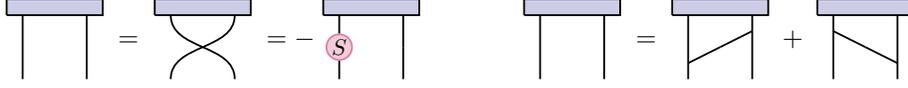

\begin{center}
$\tp{1.5cm}{\varboxtensor\varid{1}}=\tp{1.5cm}{\varboxtensor\vartrans{1}}
=-\tp{1.5cm}{\varboxtensor\twoantip{1}{0}{1}}$\hspace{3em}
$\tp{1.5cm}{\varboxtensor\varid{1}}=\tp{1.5cm}{\varboxtensor\varopT{0}{0}{1}}+\tp{1.5cm}{\varboxtensor\varopT{1}{0}{1}}$
\end{center}
\caption{Relations for $\homfunctor_2(H)$. The relation on the right is equivalent to relation (3) of Theorem \ref{thm:rank2pres} in the presence of relations (1) and (2). 
}\label{fig:rank2}
\end{figure}
\begin{theorem}\label{thm:omega2pres}
There is an isomorphism
$$\Omega_{2}(H)\cong \frac{\overline{H^{\otimes 2}}}{\langle \id-\tau,\id-S\otimes S,\id+\gamma+\gamma^2, \epsilon\otimes \id \rangle}
$$ 
where we recall $\epsilon\colon H\to \F$ is the counit.
I.e. $\Omega_{2}(H)$ is presented by $\overline{H^{\otimes 2}}$ modulo the relations
\begin{enumerate}
\item $a\otimes b=b\otimes a=S(a)\otimes S(b)$
\item $\eta(1)\otimes a=0$
\item  
$a\otimes b+ S(b)a_{(1)}\otimes a_{(2)}+ b_{(1)}\otimes S(a)b_{(2)}=0.$
\end{enumerate} 
\end{theorem}

\subsection{Computations}
We now use Theorem~\ref{thm:omega2pres} to do some calculations. Let $\omega_{2n}=\left\lceil\dfrac{2n}{3}\right\rceil$.

\begin{proposition}\label{martin:prop}
The vector space dimension of $\Omega_2(\sym(\F))_{2n}$ is given by $\omega_{2n}-1$ while $\Omega_2(\sym(\F))_{2n+1}=0$. 
 \end{proposition}
 \begin{proof}
 The odd degree case is evident. The following argument for the even degree case is due to Martin Kassabov. 
 
 Start with the observation that $\tau=\begin{bmatrix} 0&1\\1&0\end{bmatrix}$ and ${\gamma}=\begin{bmatrix}
 0&1\\-1&-1
 \end{bmatrix}
 $ generate a copy of the symmetric group $\Sigma_3$.
  Now decompose $\sym(\F)^{\otimes 2}\cong \F[x,y]$ into a direct sum of irreducible $\Sigma_3$-modules. Modding out by $1+{\gamma}+{\gamma}^2$ and $1-\tau$ kill both $1$ dimensional representations and reduce the dimension of the $2$ dimensional representation to $1$. This can be done explicitly by realizing the $2$-dimensional representation as the subspace of $\F^3$ consisting of $(t_1,t_2,t_3)$ where $t_1+t_2+t_3=0$. 
Now $\Omega_{2n}(\sym(\F))$ is the quotient of $\F[x,y]_{2n}$ by the relations $1-\tau$, $1+{\gamma}+{\gamma}^2$ and $x^{2n}=y^{2n}=0$.
The above representation theory argument shows that, $\F[x,y]_{2n}$ modulo just $1-\tau$ and $1+{\gamma}+{\gamma}^2$  will have the same dimension as the multiplicity of the two dimensional representations in $\C[x,y]_{2n}$. It's also not hard to show that further modding out by the two monomials $x^{2n}$ and $y^{2n}$ will reduce the dimension by $1$. Now to calculate the $\Sigma_3$ decomposition, we calculate the character. It's not too hard to see that the character is given by $\chi=\begin{bmatrix} 2n+1&1&1\end{bmatrix}, \begin{bmatrix} 2n+1&1&-1\end{bmatrix}$ or $\begin{bmatrix} 2n+1&1&0\end{bmatrix}$ depending on the congruence class of $n$ modulo $3$. From here it follows that the multiplicity of the two dimensional representation in $\C[x,y]_{2n}$ is $\frac{2n+1}{3}, \frac{2n+2}{3}$ or $\frac{2n}{3}$ again depending on the congruence class modulo 3. 
 \end{proof}
 For the next proposition,
 suppose that $\omega_{2k+1}=0$.
 \begin{corollary}\label{cor:omega2comp}
 We have a $\GL$-decomposition
 $$\Omega_2(\sym(V))\cong \bigoplus_{k>\ell> 0}\omega_{k-\ell}[k,\ell]_{\GL}\oplus\bigoplus_{k\geq 0} (\omega_{2k}-1)  [2k]_{\GL}.$$ 
 \end{corollary}
\begin{proof}
%We map $\Omega_2(V)$ to the quotient gotten by mapping $T(V)$ to $\sym(V)$: $$\dfrac{\sym(V)^{\otimes 2}}{\langle 1+\smat\tmat+(\smat\tmat)^2, 1-\tau\rangle+\C\{1\otimes a, a\otimes 1\}}.$$ 
We use Schur-Weyl duality:
\begin{align*}
\sym(V)^{\otimes 2}&\cong \bigoplus_{k>\ell\geq 0} \SF{(k,\ell)}(V)\otimes \SF{(k,\ell)}(\F^2)\\
& \cong \bigoplus_{k>\ell\geq 0}\SF{(k,\ell)}(V)\otimes \C[x,y]_{k-\ell}\otimes (\det)^\ell\\
\end{align*}
Modding out by $1+\smat\tmat+(\smat\tmat)^2$ and $1-\tau$ corresponds to modding out $\C[x,y]_{k-\ell}$ by $1+\smat\tmat+(\smat\tmat)^2$ and $1-(-1)^{\ell}\tau$. This has the same effect on the fundamental $\Sigma_3$ representations as in the previous proposition, so the dimension will be $\lceil \frac{k-\ell}{3}\rceil$. 

When we additionally mod out by $1\otimes a$ and $a\otimes 1$, these will map into the $\SF{(k)}(V)\otimes  \SF{(k)}(\F^2)$ summands, and so will only reduce their dimension in the total answer. 
\end{proof}

Note that $T(V)=U(\mathsf L(V))$ is the universal enveloping algebra of the free Lie algebra. Let $\mathsf{L}_{(2)}(V)\cong V\oplus\ext^2 V$ be the free metabelian Lie algebra. Then there is a map $T(V)\twoheadrightarrow \mathsf{L}_{(2)}(V)$ and a corresponding surjection $\Omega_2(T(V))\twoheadrightarrow \Omega_2(\mathsf{L}_{(2)}(V))$. The following corollary thus gives large families of representations in the Johnson cokernel, and implies Theorem~\ref{thm:omega2intro} stated in the introduction.
\begin{corollary}
$$\Omega_2(U(\mathsf{L}_{(2)}(V)))\cong \bigoplus_{k>\ell> 0}\overline{\SF{(k,\ell)}(\mathsf{L}_{(2)}(V))}^{\otimes \omega_{k-\ell}}\oplus\bigoplus_{k\geq 0} \overline{\SF{(2k)}(\mathsf{L}_{(2)}(V))}^{\otimes(\omega_{2k}-1)},$$
where $\overline{\SF{(\lambda)}(\mathsf{L}_{(2)}(V)))}$ is the quotient of the Schur functor by the adjoint action.
\end{corollary}
\begin{proof}
Let $H=U(\mathfrak g)$ be the universal enveloping algebra of a Lie algebra $\mathfrak g$. The PBW theorem gives a coalgebra isomorphism with $\sym(\mathfrak g)$. Thus $H^{\otimes 2}$ is isomorphic (as a coalgebra) to $S(\mathfrak g \otimes \F^2)$.  This has an obvious $\GL_2(\Z)$ action coming from its action on $\F^2$ (which coincides with the definition from section~\ref{sec:action} when you take the usual Hopf algebra structure on $S(\mathfrak g)$.
Therefore $\overline{H^{\otimes 2}}$ has two different actions of $\GL_2(\Z)=\Out(F_2)$. One is the action constructed in section \ref{sec:action}, and the other is the action coming from the PBW isomorphism.
In \cite{CK} it is shown that these two actions coincide if commutators of the form $[X,[X,Y]]$ vanish in $\mathfrak g$. In particular this holds for $\mathsf L_{(2)}$. So we calculate
\begin{align*}
U(\mathsf L_{(2)})^{\otimes 2}\cong
\sym(\mathsf L_{(2)})^{\otimes 2}&\cong \bigoplus_{k>\ell\geq 0} \SF{(k,\ell)}(\mathsf{L_{(2)}})\otimes \SF{(k,\ell)}(\F^2)\\
& \cong \bigoplus_{k>\ell\geq 0}\SF{(k,\ell)}(\mathsf L_{(2)})\otimes \C[x,y]_{k-\ell}\otimes (\det)^\ell\\
\end{align*}
Passing to the quotient  $\overline{U(\mathsf L_{(2)})^{\otimes 2}}$ results in taking the quotient of the Schur functors $\SF{(k,\ell)}(\sf L_{(2)})$ by the adjoint action of $\mathsf L_{(2)}$. See \cite{CK}.  The rest of the argument is the same as Corollary~\ref{cor:omega2comp}.
\end{proof}

\subsection{Proofs of Theorems~\ref{thm:rank2pres} and ~\ref{thm:omega2pres}}\label{sec:proofsrank2}
We give two different proofs of Theorem~\ref{thm:rank2pres}, one by directly computing what happens in the graph complex and one computing $H^1(\Out(F_2);\overline{H^{\otimes 2}})\cong \homfunctor_2(H)$ from the group cohomology chain complex. 
\begin{proof}[First proof of Theorem~\ref{thm:rank2pres}:]
$\cG_{H\Lie,1}^{(2)}$ is generated by graphs of the form
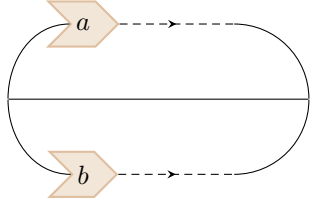
\begin{center}
\begin{minipage}{4cm}
\begin{tikzpicture}
\node[empty](A){};
\node[empty](B)[below of=A]{};
\node[empty](C)[above of=A]{};
\node[hopf](D)[right of =C]{$a$};
\node[hopf](E)[right of =B]{$b$};
\node[empty](F)[right of =D, node distance=2cm]{};
\node[empty](G)[right of =E, node distance=2cm]{};
\node[empty](H)[right of =F]{};
\node[empty](I)[right of =G]{};
\node[empty](J)[below of =H]{};
\begin{scope}[decoration={markings,mark = at position 0.5 with {\arrow{stealth}}}]
\draw[densely dashed, postaction=decorate] (D) to (F);
\draw[densely dashed, postaction=decorate] (E) to (G);
\end{scope}
\draw (A) to[out=90,in=180, thick] (D);
\draw (A) to[out=270,in=180, thick] (E);
\draw (F) to[out=0,in=90, thick] (J);
\draw (G) to[out=0,in=270, thick] (J);
\draw (A) to[thick] (J);
\end{tikzpicture}
\end{minipage}
\end{center}
which we identify with $a\otimes b\in \overline{H^{\otimes 2}}$. (The element $a\otimes b$ is only well defined up to conjugation.) The other type of graph, the eyeglass graph, is a linear combination of two of these via an IHX relation.
The various symmetries of the graph lead to the relations $a\otimes b=b\otimes a=S(a)\otimes S(b)$. Now $\partial(\cG_{H\Lie,2}^{(2)})$ is generated by the boundaries of the following two types of graphs:
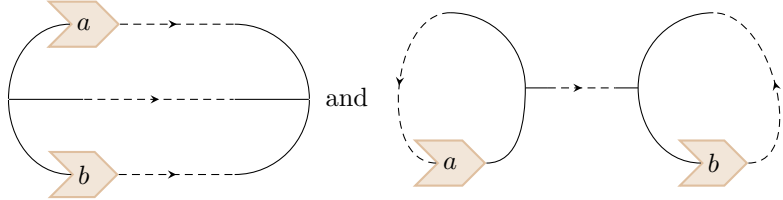
\begin{center}
\begin{minipage}{4.1cm}
\begin{tikzpicture}
\node[empty](A){};
\node[empty](B)[below of=A]{};
\node[empty](C)[above of=A]{};
\node[hopf](D)[right of =C]{$a$};
\node[hopf](E)[right of =B]{$b$};
\node[empty](F)[right of =D, node distance=2cm]{};
\node[empty](G)[right of =E, node distance=2cm]{};
\node[empty](H)[right of =F]{};
\node[empty](I)[right of =G]{};
\node[empty](J)[below of =H]{};
\node[empty](K)[below of = D]{};
\node[empty](L)[below of = F]{};
\begin{scope}[decoration={markings,mark = at position 0.5 with {\arrow{stealth}}}]
\draw[densely dashed, postaction=decorate] (D) to (F);
\draw[densely dashed, postaction=decorate] (E) to (G);
\draw[densely dashed, postaction=decorate] (K) to (L);
\end{scope}
\draw (A) to[out=90,in=180, thick] (D);
\draw (A) to[out=270,in=180, thick] (E);
\draw (F) to[out=0,in=90, thick] (J);
\draw (G) to[out=0,in=270, thick] (J);
\draw (A) to[thick] (K);
\draw (L) to[thick] (J);
\end{tikzpicture}
\end{minipage}
and
\begin{minipage}{6cm}
\begin{tikzpicture}
\coordinate(a) at (0,0);
\coordinate(d) at (1.5,0);
\coordinate(b) at (-1,1);
\coordinate(c) at (-1,-1);
\coordinate (e) at (2.5,1);
\coordinate (f) at (2.5,-1);
\coordinate (g) at (.3,0);
\coordinate (h) at (1.2,0); 
\node[hopf, at=(c)](C){$a$};
\node[hopf, at=(f)](F){$b$};
\begin{scope}[decoration={markings,mark = at position 0.5 with {\arrow{stealth}}}]
\draw[densely dashed, postaction=decorate] (b) to[out=180, in =180](C);
\draw[densely dashed, postaction=decorate] (F) to[out=0, in=0 ](e);
\draw[densely dashed, postaction=decorate] (g) to (h);
\end{scope}
\draw (a) to[thick] (g);
\draw (h) to[thick] (d);
\draw (a) to[thick, out=90, in =0] (b);
\draw (a) to[thick, out=270, in =0] (C);
\draw (d) to[thick, out=270, in =180] (F);
\draw (d) to[thick, out=90, in =180] (e);
\end{tikzpicture}
\end{minipage}
\end{center}
The boundary of the first graph has three terms corresponding to contracting along each of the three dashed edges. Contracting along the middle edge gives $a\otimes b$. To contract along the other two edges, we use the fact that
$$\begin{minipage}{4cm}
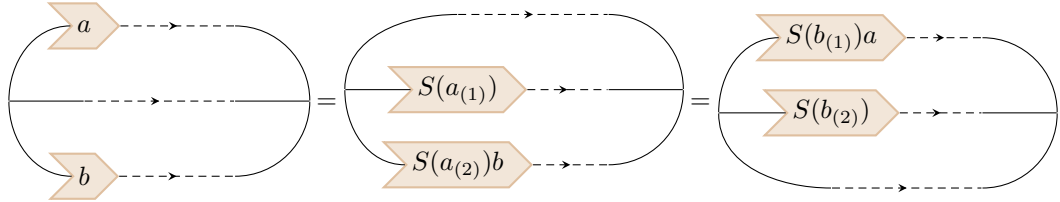
\begin{tikzpicture}
\node[empty](A){};
\node[empty](B)[below of=A]{};
\node[empty](C)[above of=A]{};
\node[hopf](D)[right of =C]{$a$};
\node[hopf](E)[right of =B]{$b$};
\node[empty](F)[right of =D, node distance=2cm]{};
\node[empty](G)[right of =E, node distance=2cm]{};
\node[empty](H)[right of =F]{};
\node[empty](I)[right of =G]{};
\node[empty](J)[below of =H]{};
\node[empty](K)[below of = D]{};
\node[empty](L)[below of = F]{};
\begin{scope}[decoration={markings,mark = at position 0.5 with {\arrow{stealth}}}]
\draw[densely dashed, postaction=decorate] (D) to[densely dashed, postaction=decorate](F);
\draw[densely dashed, postaction=decorate] (E) to[densely dashed, postaction=decorate](G);
\draw[densely dashed, postaction=decorate] (K) to[densely dashed, postaction=decorate](L);
\end{scope}
\draw (A) to[out=90,in=180, thick] (D);
\draw (A) to[out=270,in=180, thick] (E);
\draw (F) to[out=0,in=90, thick] (J);
\draw (G) to[out=0,in=270, thick] (J);
\draw (A) to[thick] (K);
\draw (L) to[thick] (J);
\end{tikzpicture}
\end{minipage}=
\begin{minipage}{4.5cm}
\begin{tikzpicture}
\node[empty](A){};
\node[empty](B)[below of=A]{};
\node[empty](C)[above of=A]{};
\node[empty](D)[right of =C, node distance=1.5cm]{};
\node[hopf](E)[right of =B, node distance=1.5cm]{$S(a_{(2)})b$};
\node[empty](F)[right of =D, node distance=2cm]{};
\node[empty](G)[right of =E, node distance=2cm]{};
\node[empty](H)[right of =F]{};
\node[empty](I)[right of =G]{};
\node[empty](J)[below of =H]{};
\node[hopf](K)[below of = D]{$S(a_{(1)})$};
\node[empty](L)[below of = F]{};
\begin{scope}[decoration={markings,mark = at position 0.5 with {\arrow{stealth}}}]
\draw[densely dashed, postaction=decorate] (D) to[densely dashed, postaction=decorate](F);
\draw[densely dashed, postaction=decorate] (E) to[densely dashed, postaction=decorate](G);
\draw[densely dashed, postaction=decorate] (K) to[densely dashed, postaction=decorate](L);
\end{scope}
\draw (A) to[out=90,in=180, thick] (D);
\draw (A) to[out=270,in=180, thick] (E);
\draw (F) to[out=0,in=90, thick] (J);
\draw (G) to[out=0,in=270, thick] (J);
\draw (A) to[thick] (K);
\draw (L) to[thick] (J);
\end{tikzpicture}
\end{minipage}=
\begin{minipage}{4.5cm}
\begin{tikzpicture}
\node[empty](A){};
\node[empty](B)[below of=A]{};
\node[empty](C)[above of=A]{};
\node[hopf](D)[right of =C, node distance=1.5cm]{$S(b_{(1)})a$};
\node[empty](E)[right of =B, node distance=1.5cm]{};
\node[empty](F)[right of =D, node distance=2cm]{};
\node[empty](G)[right of =E, node distance=2cm]{};
\node[empty](H)[right of =F]{};
\node[empty](I)[right of =G]{};
\node[empty](J)[below of =H]{};
\node[hopf](K)[below of = D]{$S(b_{(2)})$};
\node[empty](L)[below of = F]{};
\begin{scope}[decoration={markings,mark = at position 0.5 with {\arrow{stealth}}}]
\draw[densely dashed, postaction=decorate] (D) to[densely dashed, postaction=decorate](F);
\draw[densely dashed, postaction=decorate] (E) to[densely dashed, postaction=decorate](G);
\draw[densely dashed, postaction=decorate] (K) to[densely dashed, postaction=decorate](L);
\end{scope}
\draw (A) to[out=90,in=180, thick] (D);
\draw (A) to[out=270,in=180, thick] (E);
\draw (F) to[out=0,in=90, thick] (J);
\draw (G) to[out=0,in=270, thick] (J);
\draw (A) to[thick] (K);
\draw (L) to[thick] (J);
\end{tikzpicture}
\end{minipage}
$$
So the boundary becomes $a\otimes b+S(a_{(1)})\otimes S(a_{(2)})b+S(b_{(1)})a\otimes S(b_{(2)})$ which is equivalent to the third relation, in the presence of the relation $a\otimes b=b\otimes a$.

The boundary of the second graph has only one term, which is equal by an IHX relation to $-a\otimes S(b)-a\otimes b$. This gives the relation $a\otimes b=-S(a)\otimes b$, and together with symmetry derived above, its consequence $a\otimes b=S(a)\otimes S(b)$.

This completes the first proof.
\end{proof}

\begin{proof}[Second proof of Theorem~\ref{thm:rank2pres}:]
The statement follows immediately from Proposition~\ref{prop:glpres} and the fact that the natural map $\Out(F_2)\to\GL_2(\Z)$ is an isomorphism. 
\end{proof}

\begin{proof}[Proof of Theorem~\ref{thm:omega2pres}]

We follow the first proof of Theorem~\ref{thm:rank2pres}, but we should only mod out by the first type of boundary, and the second type of boundary only when one of the two loops does not have an element of $H$ adorning it. By an IHX relation, this becomes $2(1\otimes b)=0$.
\end{proof}

\section{Cohomology of $\GL_3(\mathbb Z)$}
In this section, we prove some preliminary results about the cohomology of $\GL_3(\Z)$ before considering the rank $3$ spaces $\homfunctor_3(H)$ and $\Omega_3(H)$.

\newcommand{\grp}{\GL_3(\mathbb Z)}
Let $W=\F\{x,y,z\}$ be the standard representation of $\GL_3(\mathbb Z)$. Then the irreducible finite dimensional $\GL_3(\F)$ modules are all of the form $\Psi_{(a,b,c)}=\SF{(a,b,c)}(W)\otimes (\det)^k$ for a partition $\lambda=(a,b,c)$ and $k\in\mathbb Z$. Furthermore $\Psi_{(a,b,c)}=\Psi_{(a+k,b+k,c+k)}\otimes (\det)^{-k}$, allowing us to extend  $\Psi_{(a,b,c)}$ to arbitrary triples $a\geq b\geq c$ of possibly negative integers. Note that all representations are represented by $\Psi_{(a,b,0)}\otimes (\det)^k$ for some $k$. We also mention that the dual representation satisfies $\Psi^*_{(a,b,c)}=\Psi_{(-c,-b,-a)}$.

In this section, we shall be concerned with calculating $H^3(\GL_3(\mathbb Z); \Psi_{(a,b,0)}\otimes (\det)^k)$. A first observation is that $-I_3\in \GL_3(\mathbb Z)$ is central, and if $-I_3$ acts nontrivially on an irreducible module $M$, then, thinking of $-I_3$ as generating a copy of $\mathbb Z_2$ inside $\GL_3(\Z)$, we have $H^k(\GL_3(\Z);M)\cong H^k(\mathrm{PGL}_3(\Z);H^0(\Z_2;M))\cong
 H^k(\mathrm{PGL}_3(\Z);M^{\mathbb Z_2})= H^k(\mathrm{PGL}_3(\Z);0)=0$.
 
As a result, $H^3(\GL_3(\mathbb Z);\Psi_{(a,b,c)}\otimes (\det)^k)=0$ if $a+b+c+k$ is odd. So we now confine our attention to the even case and assume we have a representation $\Psi_{(a,b,c)}$ where $a+b+c$ is even.

Let $M$ be an arbitrary irreducible representation and let us temporarily switch to homology cf. Remark~\ref{rem:dual}. Then there is a decomposition
$$H_3(\grp;M)=H^\partial_3(\grp;M)\oplus H^{\mathrm{cusp}}_3(\grp;M)$$
into \emph{boundary} and \emph{cuspidal} homology respectively \cite{Ash-Grayson-Green}. It follows from a theorem of Borel and Wallach \cite[II.6.12]{Borel-Wallach} that $H^{\mathrm{cusp}}_3(\grp;M)=0$ unless $M$ is self dual. The only self dual module in even degree is $\Psi_{(2g,g,0)}$ for $g$ even. Let $s_w$ be the dimension of the space of cusp forms of weight $w$.
It is known that for $g$ even $\dim[H^{\mathrm{cusp}}_3(\grp;\Psi_{(2g,g,0)})]\geq s_{g+2}$  arising from the symmetric square construction \cite{ash-pollack}. In fact, it is conjectured that
\begin{conjecture}[Ash-Pollack]\label{conj:ash-pollack}
$$\dim[H^{\mathrm{cusp}}_3(\grp;\Psi_{(2g,g,0)})]= s_{g+2}$$
\end{conjecture}
For the boundary homology, in \cite{allison-ash-conrad} the authors construct three subsets of $H^\partial_3(\grp;M)$ and conjecture that they span all of it. They remark that a proof that the subsets span should follow from arguments similar to those found in Lee and Schwermer's paper \cite{LS}. To explain their result, consider the two subgroups $P$ and $Q$ of $\grp$ which are the stabilizers of the lines $(0,0,*)^t$ and $(*,0,0)$ respectively. Let $B=P\cap Q$ be the group of lower triangular matrices. The unipotent radicals of $P$ and $Q$ can be written as
$$U_P=\begin{pmatrix}1&0&0\\0&1&0\\
*&*&1\end{pmatrix}
\hspace{2em}
U_Q=\begin{pmatrix}1&0&0\\
*&1&0\\
*&0&1\end{pmatrix}
$$
Paraphrasing their notation, define homomorphisms $L_P,L_Q\colon \GL_2(\Z)\to \grp$  via
$$X\mapsto \begin{pmatrix}X&0\\0&1\end{pmatrix} \hspace{2em} X\mapsto \begin{pmatrix}1&0\\0&X\end{pmatrix}$$ 
respectively. Finally, define the antisymmetrizing operator $A=\sum_{\sigma\in \Sigma_3} (-1)^{|\sigma|}\sigma$ where the symmetric group $\Sigma_3$ is realized as permutation matrices in $\grp$. 

Finally, recall the three defining equations for $H_1(\GL_2(\Z);M\otimes(\det))$:
\begin{equation}
m(1+\smat\tmat+(\smat\tmat)^2)=m(1+\smat)=m(1+\tau)=0.\tag{$\dagger$}
\end{equation}
Now they define three subsets of $M$: 
\begin{enumerate}
\item $\{A(m)\in M\,|\,\ m\in M^B\}$,  
\item $\{A(m)\,|\, m \in M^{U_P} \text{ and $m$ satisfies ($\dagger$) under $L_P$}\}$
\item $\{A(m)\,|\, m \in M^{U_Q} \text{ and $m$ satisfies ($\dagger$) under $L_Q$}\}$
\end{enumerate}

\begin{theorem}[Allison-Ash-Conrad]
$H_3^\partial(\grp;M)$ can naturally be regarded as a subspace of $M$, and the three subsets above are contained in $H_3^\partial(\grp;M)$. Conjecturally these are everything.
\end{theorem}
\begin{remark}
The matrix $h_2=\begin{bmatrix}0&-1\\ 1&-1\end{bmatrix}$ which they use satisfies $\delta h_2\delta=st$. 
 Now $\delta=\tau s$, so the relations $\id+\tau=\id+s=0$ imply that $m\cdot\delta=m$.
So in the presence of these relations, $m\cdot (\id+st+(st)^2)=0$ is an equivalent relation to $m\cdot( \id+h_2+h_2^2)=0$.
%In fact, the relations $m(1+\smat\tmat+(\smat\tmat)^2)$ and $m(\id+h_2+h_2^2)$ are both equivalent to:
%$$\tp{1.5cm}{\varboxtensor\varid{1}}-
%\tp{1.5cm}{\varboxtensor\varopT{0}{0}{1}}-\tp{1.5cm}{\varboxtensor\varopT{1}{0}{1}}=0.$$
\end{remark}

Our first task is to identify the invariants $M^B,M^{U_P}$ and $M^{U_Q}$. We take as our model for $M=\Psi_{(a,b,c)}$ the $\grp$ submodule of $W^{\otimes {a+b+c}}$ generated by $f=x^{a-b}[x,y]^{b-c}(A(xyz))^c$. In the following lemma,  $\sym^k(x,y)$ denotes the $k$th symmetric power of the $\F$-vector space spanned by $x,y$.
\begin{lemma}\
\begin{enumerate}
\item $M^B$ is $0$ unless $a,b,c$ are even, in which case $M^B\cong \F\{f\}$. 
\item $M^{U_P}= \sym^{a-b}(x,y)[x,y]^{b-c}(A(xyz))^c\cong\SF{(a,b)}(\F\{x,y\})$.
\item $M^{U_Q}=x^{a-b}\sym^{b-c}([x,y],[x,z])(A(xyz))^c\cong \SF{(b,c)}(\F\{y,z\})$.
\end{enumerate}
\end{lemma} 
\begin{proof}
It is easy to check that $\F\{f\}$ is $B$-invariant,
$\sym^{a-b}(x,y)[x,y]^{b-c}(A(xyz))^c$ is $U_P$ invariant, and that  $x^{a-b}\sym^{b-c}([x,y],[x,z])(A(xyz))^c$ is $U_Q$ invariant. A somewhat tedious caclulation shows that these actually generate the entire invariant subspaces. However, the obvious containment is all that will be needed to prove Theorem~\ref{thm:bdrycalc}, so we omit the proof of equality.
\end{proof}
\begin{lemma}\
\begin{enumerate}
\item Suppose $b>c$. Then $\ker A|_{M^{U_P}}=\{m\in M^{U_P} \,|\, (1-L_P(\tau))m=0\}$. 
\item Suppose $a>b$. Then $\ker A|_{M^{U_Q}}=\{m\in M^{U_Q} \,|\, (1-L_Q(\tau))m=0\}$
\item $A[M^{U_P}]\cap A[M^{U_Q}]$ is at most one-dimensional  and is spanned by $A(f)$. $A(f)\neq 0$ iff one of the three following conditions is satisfied
 \begin{enumerate}
 \item $a>b>c$.
 \item $a=b$ and $b$ is odd.
\item $b=c$ and $b$ is odd.
\end{enumerate}
\end{enumerate}
\end{lemma}
\begin{proof}
For the first claim, let $m=\varphi(x,y)[x,y]^{b-c}(A(xyz))^c$, and
observe that 
$$A\cdot m\in \F[x,y] (A(xyz))^c\oplus \F[x,z] (A(xyz))^c\oplus \F[y,z] (A(xyz))^c.$$ 
Suppose $A\cdot m=0$. Projecting to each component (assuming $b>c$) all yield the same condition on $m$, namely that $m=L_P(\tau)\cdot m$.

For the second claim, consider $m=(x^{a-b}\varphi([x,y],[x,z])(A(xyz))^c$ and
 $$A\cdot m\in x^{a-b}\F[x,y,z] (A(xyz))^c\oplus y^{a-b}\F[x,y,z] (A(xyz))^c\oplus z^{a-b}\F[x,y,z] (A(xyz))^c\}.$$ Suppose that $A\cdot m=0$. Looking at each component again yields the desired equation $m=L_Q(\tau)\cdot m$, assuming $a>b$.

For the third claim, suppose $$A(\varphi(x,y)[x,y]^{b-c}(A(xyz))^c)=A(x^{a-b}\psi([x,y],[x,z])(A(xyz))^c).$$ Project to the subspace $\F[x,y](A(xyz))^c$. Note that if $A(x^{a-b}\psi([x,y],[x,z])(A(xyz))^c)\neq 0$, then $\psi(A,B)=rA^{b-c}$ or $\psi(A,B)=r B^{b-c}$. Otherwise  every term will involve all three generators. Suppose $ \psi(A,B)=rA^{b-c}$. Then we have an equation 
$$\varphi(x,y)[x,y]^{b-c}-(-1)^c\varphi(y,x)[y,x]^{b-c}=
rx^{a-b}[x,y]^{b-c}-r(-1)^cy^{a-b}[y,x]^{b-c}$$
which implies
$$\varphi(x,y)-(-1)^{b}\varphi(y,x)=rx^{a-b}-(-1)^{b}ry^{a-b}.$$
A similar calculation holds for $\psi(A,B)=rB^{b-c}$. 

Let 
\begin{align*}
\xi=&A(x^{a-b}[x,y]^{b-c}(A(xyz))^c)\\
=&(x^{a-b}[x,y]^{b-c}-(-1)^by^{a-b}[x,y]^{b-c}-(-1)^bz^{a-b}[y,z]^{b-c}-(-1)^cx^{a-b}[x,z]^{b-c}+\\
&y^{a-b}[y,z]^{b-c}+(-1)^{b-c}z^{a-b}[x,z]^{b-c})(A(xyz))^c
\end{align*}
Then our calculations imply that $A[M^{U_P}]\cap A[M^{U_Q}]$ is at most one dimensional, spanned by the vector $\xi$. If $a>b>c$, then $\xi\neq 0$. If $a=b$  then $\xi\neq 0\leftrightarrow b\text{ is odd}$. If $b=c$, then again $\xi\neq 0\leftrightarrow b\text{ is odd}$. 
\end{proof}

\begin{lemma}\label{lem:alphabeta}\
\begin{enumerate}
\item If $b=c$, then $\ker A|_{M^{U_P}}\cap\ker(1+\tau)$ is spanned by $\alpha=(x^{a-b}-y^{a-b})A(xyz)^b$.
\item If $a=b$, then $\ker A|_{M^{U_Q}}\cap\ker(1+\tau)$ is spanned by $\beta =x^{a-b}([x,y]^{b-c}-[x,z]^{b-c})(A(xyz))^c$.
\end{enumerate}
\end{lemma}
\begin{proof}
When $b=c$, if $\varphi(x,y)(A(xyz))^c\in \ker A|_{M^{U_P}}\cap\ker(1+\tau)$,
we arrive at the equation $\varphi(x,y)+\varphi(y,z)+\varphi(z,x)=0$. If $\varphi(x,y)$ contains any mixed terms $x^uy^v$ ($u,v\geq 1$), then cancellation is impossible in this equation. Thus $\varphi(x,y)$ is a linear combination of $x^{a-b}$ and $y^{a-b}$, and by antisymmetry it has to be a multiple of $\alpha$.

The proof for $\beta$ is similar.
\end{proof}
Now we put our calculations together to calculate $H_3^\partial(\grp;\Psi_{(a,b,c)})$.
\begin{theorem}\label{thm:bdrycalc}
Let $a\geq b\geq c\geq 0$ be a triplet of nonnegative integers. Then
$$\dim(H_3^{\partial}(\grp;\Psi_{(a,b,c)}))\geq s_{a-b+2}+s_{b-c+2}+\epsilon_{a,b,c}$$
where $$\epsilon_{a,b,c}=\begin{cases}
1&\text{if } a>b>c \text{ and } a,b,c\text{ are all even}\\
0&\text{otherwise}
\end{cases}.$$
\end{theorem}
\begin{proof}
 If we look at type (2) elements, the subset of $M^{U_P}\cong \Psi_{(a,b)}$ satisfying ($\dagger$) is isomorphic to $H_1(\GL_2(\Z);\Psi_{(a,b)}\otimes (\det))$. Similarly for type (3) elements we have that the subset of $M^{U_Q}\cong\Psi_{(b,c)}$ satisfying ($\dagger$) is isomorphic to $H_1(\GL_2(\Z);\Psi_{(b,c)}\otimes (\det))$. Now recall that 
$$H_1(\GL_2(\Z);\Psi_{(a,b)}\otimes (\det))\cong\begin{cases} \mathcal M_{a-b+2} &\text{if $b$ is even}\\
\mathcal S_{a-b+2}&\text{if $b$ is odd}.\end{cases}$$
Moreover if $b$  is even (and so $a,c$ are even in the only interesting case) then consider the elements
$\alpha\in M^{U_P}$
and
$\beta\in M^{U_Q},$ defined in Lemma~\ref{lem:alphabeta}.
These both satisfy $(\dagger)$ so are actually elements of $H_1(\GL_2(\Z);\Psi_{(a,b)}\otimes (\det))$ and $H_1(\GL_2(\Z);\Psi_{(b,c)}\otimes (\det))$ respectively. It's also easy to see that $A(\alpha)$ and $A(\beta)$ are both in the subspace generated by $A(f)=\xi$. Thus if $a>b>c$ are all even, $$\dim(H_3^\partial(\grp;\Psi_{(a,b,c)}))\geq s_{a-b+2}+s_{b-c+2}+1$$ since $A$ is injective on $H_1(\GL_2(\Z);\Psi_{(a,b)}\otimes (\det))$, $H_1(\GL_2(\Z);\Psi_{(b,c)}\otimes (\det))$ and $M^{B}$, and they all overlap in a one-dimensional subspace. 

Suppose that $a=b>c$ are all even. Then $A(f)=0$. The kernel of $A$ on $H_1(\GL_2(\Z);\Psi_{(b,c)}\otimes (\det))$ is spanned by $\beta$ so $\dim(H_3^\partial(\grp;\Psi_{(a,b,c)}))=s_{b-c+2}$.

Suppose that $a>b=c$ are all even. Then $A(f)=0$ and the kernel of $A$ on $H_1(\GL_2(\Z);\Psi_{(a,b)}\otimes (\det))$ is spanned by $\alpha$. So $\dim(H_3^\partial(\grp;\Psi_{(a,b,c)}))\geq s_{a-b+2}.$

Finally we need to consider the case when two of $a,b,c$ are odd and one is even. Then all type one classes are $0$, since elements of $M^B$ need to be even in each degree. Also, the other two pieces are images of $H_1(\GL_2(\Z);\Psi_{(a,b)}\otimes (\det))$ and $H_1(\GL_2(\Z);\Psi_{(b,c)}\otimes (\det))$ under $A$, and there is only $\GL_2(\Z)$ homology in even degree. So if $b$ is even, and $a,c$ are odd then $H_3^\partial(\grp;\Psi_{(a,b,c)})=0$. 

Suppose that $a,b$ are odd and $c$ is even. Then we are looking at the image of $\mathcal{S}_{a-b+2}$ under $A$. We showed this injects under the assumption that $b>c$ which clearly holds here. 
So the dimension is at least $s_{a-b+2}$.

Similarly if $a$ is even and $b,c$ are odd, we get $s_{b-c+2}$.
\end{proof}

Now, adding to this the cuspidal homology, which is only nontrivial when $\Psi_{(a,b,c)}$ is self dual, i.e. $a-b=b-c$, we see
\begin{corollary}\label{cor:hom3}
Let $a\geq b\geq c\geq 0$ be a triplet of nonnegative integers. Then
$$\dim(H_3(\grp;\Psi_{(a,b,c)}))\geq s_{a-b+2}+s_{b-c+2}+\epsilon_{a,b,c}+\delta_{a,b,c}$$
where $$\delta_{a,b,c}=\begin{cases}
s_{a-b+2}&\text{if } a-b=b-c\\
0&\text{otherwise}
\end{cases}.$$
\end{corollary}
In particular, for the self dual module $\Psi_{(2g,g,0)}$ gives homology of dimension at least $3s_{g+2}+1$, in accordance with the formula given in \cite{ash-pollack}.
Conjecturally, the inequality in the corollary is an equality.

Returning now to cohomology, we see that 
\begin{theorem}
There is an injection
$$
 \bigoplus_{a+b+c=2k}\SF{(a,b,c)}(V)\otimes(\mathcal S_{a-b+2}\oplus \mathcal S_{b-c+2}\oplus D_{a,b,c}\oplus E_{a,b,c})\hookrightarrow H^3(\GL_3(\Z);\sym(V)^{\otimes 3})_{2k}
$$
where $E_{a,b,c}=\F$ if $a>b>c$ are all even and is equal to $0$ otherwise, and $D_{a,b,c}=\mathcal S_{a-b}$ if $a-b=b-c$, and is equal to $0$ otherwise.
\end{theorem}
\begin{proof}
Use Remark~\ref{rem:dual} and the fact that $\Psi^*_{(a,b,c)}\cong \Psi_{(2k-c,2k-b,2k-a)}$ as a $\GL_3(\Z)$ module. The sum $s_{a-b+2}+s_{b-c+2}+\epsilon_{a,b,c}+\delta_{a,b,c}$ is invariant under the transformation $(a,b,c)\mapsto(2k-c,2k-b,2k-a)$. 
\end{proof}

 \section{Rank 3}
 In this section we give presentations for $\homfunctor_3(H)$ and $\Omega_3(H)$ and use these to do explicit computations in the case $H=\sym(V)$. The proofs of the presentations are deferred to section \ref{sec:rank3proofs}.
 
 \subsection{Presentations for $\homfunctor_3(H)$ and $\Omega_3(H)$}
 In the following, let $\sigma_{ij}\colon H^{\otimes 3}\to H^{\otimes 3}$ transpose the $i$th and $j$th factors. Let $E\colon H^{\otimes 3}\to H^{\otimes 3}$ be defined by $E=(\id\otimes m\otimes \id)\circ(\Delta\otimes\id\otimes\id)$, and similarly $F=(m\otimes\id\otimes\id)\circ(\id\otimes\Delta\otimes\id)$.
\begin{theorem}\label{thm:rank3pres}
$\homfunctor_3(H)\cong H^3(\Out(F_3);\overline{H^{\otimes 3}})$ is isomorphic to the quotient of $\overline{H^{\otimes 3}}$ by the images of the following operators (operating on the right):
\begin{enumerate}
\item %$\id+\sigma_{12}(S\otimes S\otimes S)$ 
{$\id+(S\otimes S\otimes S)\sigma_{12}$}
\item %$\id+\sigma_{13}-\sigma_{23}\sigma_{12}-\sigma_{12}$ 
{$\id+\sigma_{13}-\sigma_{12}\sigma_{23}-\sigma_{12}$ }
\item %$(\id+S)\otimes\id\otimes\id+\sigma_{23}(\id+S)\otimes\id\otimes\id$
{$((\id+S)\otimes\id\otimes\id)(\id+\sigma_{23})$}
\item %$\id+\sigma_{23}E+\sigma_{13}F-\sigma_{12}-\sigma_{13}F\sigma_{12}-\sigma_{23}E\sigma_{12}$
{$\id+E\sigma_{23}+F\sigma_{13}-\sigma_{12}-\sigma_{12}F\sigma_{13}-\sigma_{12}E\sigma_{23}$}
\item %$(\id+\sigma_{23}E+\sigma_{13}F)((\id+S)\otimes \id\otimes \id)$
{$((\id+S)\otimes \id\otimes \id)(\id+E\sigma_{23}+F\sigma_{13})$}
\item %$\sigma_{23}E\sigma_{23}(S\otimes\id\otimes \id)+\sigma_{13}F\sigma_{23}(S\otimes\id\otimes \id)+\sigma_{23}E+\sigma_{13}F$
{$(S\otimes\id\otimes \id)(\sigma_{23}E\sigma_{23}
+\sigma_{23}F\sigma_{13})
+E\sigma_{23}+ F\sigma_{13}$}
\end{enumerate}
\end{theorem}
As for the rank $2$ case, we depict these relations graphically in Figure~\ref{fig:rank2relns}.
\begin{figure}
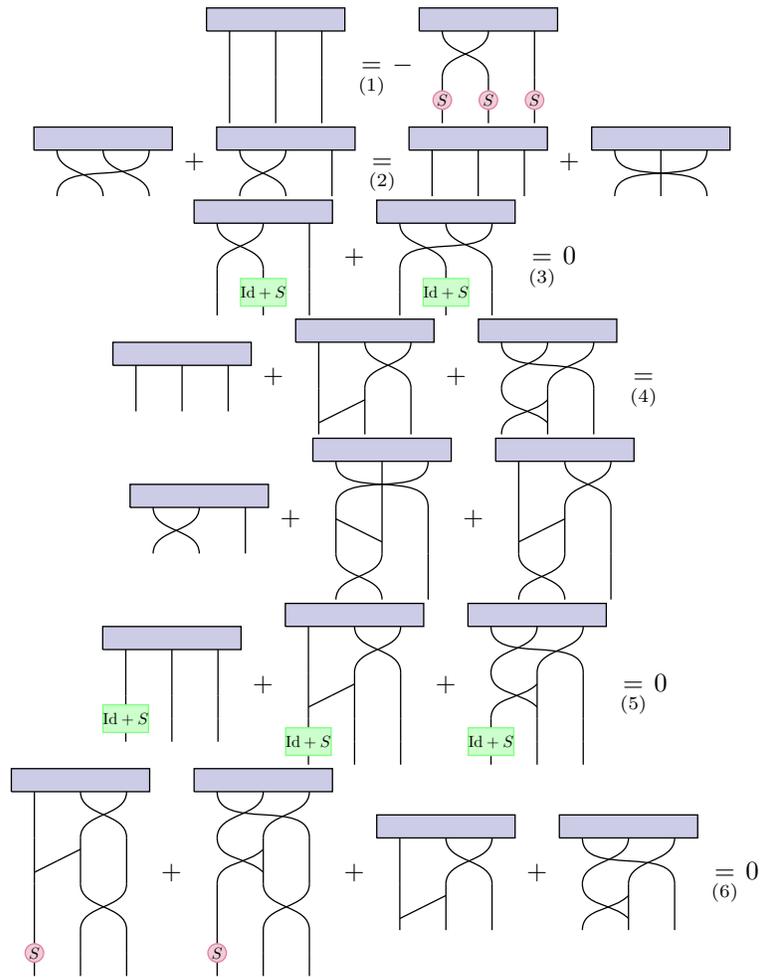

\begin{center}
$\tp{2cm}{\boxtensor\transF{1}\transF{2}}
\underset{(1)}{=}-\tp{2cm}{\boxtensor\transA{1}\antip{1}{1}{1}{2}}$\\
$\tp{2cm}{\boxtensor\transE{1}}+\tp{2cm}{\boxtensor\transA{1}}\underset{(2)}{=}\tp{2cm}{\boxtensor\transF{1}}+\tp{2cm}{\boxtensor\transB{1}}$\\
%$\tp{2cm}{\boxtensor\transA{1}}+\tp{2cm}{\boxtensor\transE{1}}\underset{(D3')}{=}-\tp{2cm}{\boxtensor\transA{1}\antip{0}{1}{0}{2}}-\tp{2cm}{\boxtensor\transE{1}\antip{0}{1}{0}{2}}$\\
$\tp{2cm}{\boxtensor\transA{1}\varantip{0}{1}{0}{2}{\id+S}}+\tp{2cm}{\boxtensor\transE{1}\varantip{0}{1}{0}{2}{\id+S}}\underset{(3)}{=}0$\\
$\tp{2cm}{\boxtensor\transF{1}}+\tp{2cm}{\boxtensor\transC{1}\opT{0}{0}{2}}+\tp{2cm}{\boxtensor\transD{1}\opX{1}{2}}\underset{(4)}{=}\tp{2cm}{\boxtensor\transA{1}}+\tp{2cm}{\boxtensor\transB{1}\opT{1}{0}{2}\transA{3}}+\tp{2cm}{\boxtensor\transC{1}\opT{0}{0}{2}\transA{3}}$\\
%$\tp{2cm}{\boxtensor\transF{1}}+\tp{2cm}{\boxtensor\transC{1}\opT{0}{0}{2}}+\tp{2cm}{\boxtensor\transD{1}\opX{1}{2}}\underset{(D4')}{=}-\tp{2cm}{\boxtensor\transF{1}\antip{1}{0}{0}{2}}-\tp{2cm}{\boxtensor\transC{1}\opT{0}{0}{2}\antip{1}{0}{0}{3}}-\tp{2cm}{\boxtensor\transD{1}\opX{1}{2}\antip{1}{0}{0}{3}}$\\
$\tp{2cm}{\boxtensor\transF{1}\varantip{1}{0}{0}{2}{\id+S}}+\tp{2cm}{\boxtensor\transC{1}\opT{0}{0}{2}\varantip{1}{0}{0}{3}{\id+S}}+\tp{2cm}{\boxtensor\transD{1}\opX{1}{2}\varantip{1}{0}{0}{3}{\id+S}}\underset{(5)}{=}0$\\
$\tp{2cm}{\boxtensor\transC{1}\opT{0}{0}{2}\transC{3}\antip{1}{0}{0}{4}}+\tp{2cm}{\boxtensor\transD{1}\opX{1}{2}\transC{3}\antip{1}{0}{0}{4}}+\tp{2cm}{\boxtensor\transC{1}\opT{0}{0}{2}}+\tp{2cm}{\boxtensor\transD{1}\opX{1}{2}}\underset{(6)}{=}0
$
\end{center}
\caption{$\homfunctor_3(H)$ is isomorphic to $\overline{H^{\otimes 3}}$ modulo the above relations.}\label{fig:rank2relns}
\end{figure}
\begin{theorem}\label{thm:omega3pres}
$\Omega_3(H)$ is presented as a quotient of $\overline{H^{\otimes 3}}$ by the images of operators (operating on the right)
\begin{enumerate}
\item $\id+(S\otimes S\otimes S)\sigma_{12}$
\item $\id+\sigma_{13}-\sigma_{12}\sigma_{23}-\sigma_{12}$
\item {$(S\otimes\id\otimes \id)(\sigma_{23}E\sigma_{23}
+\sigma_{23}F\sigma_{13})
+E\sigma_{23}+ F\sigma_{13}$}
%\item $\id+\sigma_{23}E+\sigma_{13}F-\sigma_{12}-\sigma_{13}F\sigma_{12}-\sigma_{23}E\sigma_{12}$
and by the relations
\item \label{itemA} $1\otimes a\otimes b+1\otimes b\otimes a=0$
\item \label{itemB} $a\otimes\Delta(b)=0$
\end{enumerate}
\end{theorem}

\subsection{$H=\sym(V)$}
Next we consider the specific case of $H=\sym(V)$. Decompose $\overline{\sym(V)^{\otimes 3}}\cong [\sym(V)^{\otimes 3}]_e\oplus [\sym(V)^{\otimes 3}]_o$ into even and odd degree pieces respectively. The even degree case is particularly simple:
\begin{theorem}\label{thm:evenpres}
$\homfunctor_3(\sym(V))_e\cong H^3(\Out(F_3);[\sym(V)^{\otimes 3}]_e)$ is isomorphic to the quotient of $[\sym(V)^{\otimes 3}]_e$ by the images of the operators below and also depicted in Figure~\ref{fig:even}.
\begin{enumerate}
\item $\id+\sigma_{12}$
\item $\id+\sigma_{23}$
\item $\id{-}S\otimes\id\otimes \id$
\item $E+F-\id$
\end{enumerate}
\end{theorem}
\begin{figure}
\begin{center}
$\tp{2cm}{\boxtensor\transF{1}}=-\tp{2cm}{\boxtensor\transA{1}}=-\tp{2cm}{\boxtensor\transC{1}}$\\
 $\tp{2cm}{\boxtensor\transF{1}}=\tp{2cm}{\boxtensor\antip{1}{0}{0}{1}}$\\
$\tp{2cm}{\boxtensor\transF{1}}=\tp{2cm}{\boxtensor\opT{0}{0}{1}}+\tp{2cm}{\boxtensor\opT{1}{0}{1}}$
\end{center}
\caption{$\homfunctor_3(\sym(V))_e$ is presented by $[\sym(V)^{\otimes 3}]_e$ modulo these relations.}\label{fig:even}
\end{figure}
The odd degree presentation is a bit more complicated:
\begin{theorem}\label{thm:oddpres}
$\homfunctor_3(\sym(V))_o\cong H^3(\Out(F_3);[\sym(V)^{\otimes 3}]_o)$ is isomorphic to the quotient of $[\sym(V)^{\otimes 3}]_o$ by the images of the operators below and also depicted in Figure~\ref{fig:odd}.
\begin{enumerate}
\item $\id-\sigma_{12}$
\item $(((\id+S)\otimes\id\otimes\id)\id+\sigma_{23})$
\item $(((\id+S)\otimes\id\otimes\id)\id+E\sigma_{23}+F\sigma_{13})$
\item $((\id-S)\otimes\id\otimes\id)(-\sigma_{23}E\sigma_{23}-\sigma_{23}F\sigma_{13}+E\sigma_{23}+F\sigma_{13})$
\end{enumerate}
\end{theorem}

\begin{figure}
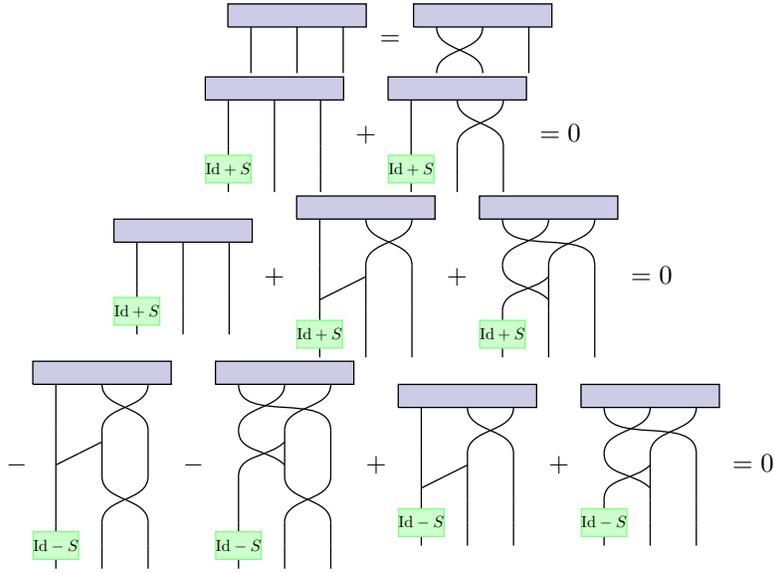

\begin{center}
$\tp{2cm}{\boxtensor\transF{1}}=\tp{2cm}{\boxtensor\transA{1}}$\\
$\tp{2cm}{\boxtensor\transF{1}\varantip{1}{0}{0}{2}{\id+S}}+\tp{2cm}{\boxtensor\transC{1}\varantip{1}{0}{0}{2}{\id+S}}=0$\\
$\tp{2cm}{\boxtensor\transF{1}\varantip{1}{0}{0}{2}{\id+S}}+\tp{2cm}{\boxtensor\transC{1}\opT{0}{0}{2}\varantip{1}{0}{0}{3}{\id+S}}+\tp{2cm}{\boxtensor\transD{1}\opX{1}{2}\varantip{1}{0}{0}{3}{\id+S}}=0$\\
$-\tp{2cm}{\boxtensor\transC{1}\opT{0}{0}{2}\transC{3}\varantip{1}{0}{0}{4}{\id-S}}-\tp{2cm}{\boxtensor\transD{1}\opX{1}{2}\transC{3}\varantip{1}{0}{0}{4}{\id-S}}+\tp{2cm}{\boxtensor\transC{1}\opT{0}{0}{2}\varantip{1}{0}{0}{3}{\id-S}}+\tp{2cm}{\boxtensor\transD{1}\opX{1}{2}\varantip{1}{0}{0}{3}{\id-S}}=0$
\end{center}
\caption{Relations for the odd degree case of $H=\sym(V)$}\label{fig:odd}
\end{figure}

\subsection{Computations}
We begin by showing that in even degrees, there is an isomorphism with $\GL_3(\Z)$ cohomology in the same degree:
\begin{proposition}
$$H^3(\Out(F_3);[\sym(V)^{\otimes 3}]_e)\cong H^3(\GL_3(\Z);[\sym(V)^{\otimes 3}]_e)$$
\end{proposition}
\begin{proof}
 From \cite{ash80} one can derive a presentation for $H^3(\SL_3(\Z);\sym(V)^{\otimes 3})$ as the quotient of $\sym(V)^{\otimes 3}$ by relations $\id+\sigma_{12}$, $\id+\sigma_{23}$, $E+F-\id$ and $S\otimes S\otimes \id-\id$. See \cite{allison-ash-conrad} where this is stated in the case of characteristic $p$ coefficients. 
To go to $\GL_3(\Z)$ we need to take the coinvariants with respect to the action of the matrix $\begin{bmatrix} -1&0&0\\0&1&0\\0&0&1\end{bmatrix}$ which gives exactly our presentation. 
\end{proof}

\begin{corollary}
Let $a+b+c$ be even. Then $H^3(\Out(F_3);\Psi_{(a,b,c)})\cong H^3(\GL_3(\Z);\Psi_{(a,b,c)}).$
\end{corollary}
\begin{proof}
Decompose $[\sym(V)^{\otimes 3}]_e=\bigoplus_{|\lambda| \text{ even}} \Psi_\lambda \otimes \SF{\lambda}(V)$, and compare the multiplicities of the irreps $\SF{\lambda}(V)$ on both sides of the isomorphism $H^3(\Out(F_3);[\sym(V)^{\otimes 3}]_e)\cong H^3(\GL_3(\Z);[\sym(V)^{\otimes 3}]_e)$.
\end{proof}

\begin{corollary}\label{cor:rank3evencomp}
$$ \bigoplus_{a+b+c=2k}\SF{(a,b,c)}(V)\otimes(\mathcal S_{a-b+2}\oplus \mathcal S_{b-c+2}\oplus D_{a,b,c}\oplus E_{a,b,c})\hookrightarrow \homfunctor_3(\sym(V))_{e}
$$
where $E_{a,b,c}=\F$ if $a>b>c$ are all even and is equal to $0$ otherwise, and $D_{a,b,c}=\mathcal S_{a-b}$ if $a-b=b-c$, and is equal to $0$ otherwise.
\end{corollary}
This establishes Theorem~\ref{thm:rank3intro} of the introduction.

Recall that $\omega_{2n}=\lceil\frac{2n}{3}\rceil$,
and $\mathcal M^{\Omega}_{2n}=\F^{\omega_{2n}}$, $\mathcal S^{\Omega}_{2n}=\F^{\omega_{2n}-1}$. Define these space to be zero in odd degrees. These are so named because we naturally have
$$\mathcal M^{\Omega}_{2n}\twoheadrightarrow\mathcal M_{2n+2},\,\,\,\,\mathcal S^{\Omega}_{2n}\twoheadrightarrow\mathcal S_{2n+2}.$$

The following theorem established Theorem~\ref{thm:omega3intro} of the introduction.
\begin{theorem}\label{thm:iota}
There is an injection
 $$\bigoplus_{a+b+c=2k} \SF{(a,b,c)}(V)\otimes\left(\mathcal S_{a-b+2}\oplus \mathcal S^\Omega_{b-c+2}\oplus D_{a,b,c}\oplus E_{a,b,c} \right)\hookrightarrow\Omega_3(\sym(V))_e.$$
\end{theorem}
\begin{remark}
By definition $\homfunctor_3(\sym(V))_e$ is a quotient of $\Omega_3(\sym(V))_e$, and this theorem manages to enlarge the $\mathcal S_{a-b+2}$ in the decomposition of $\homfunctor_3(\sym(V))_e$ into an $\mathcal S^{\Omega}_{a-b}$. It seems natural to speculate that the $\mathcal S_{b-c+2}$ can be similarly enlarged (and maybe even the $D_{a,b,c}$ symmetric square terms.) However, calculations similar to the ones carried out here for $M^{U_P}$ don't seem to extend to $M^{U_Q}$. 
%In fact, assuming that $f(x,y,z)\in M^{U_Q}$ is antisymmetric in $y,z$, satisfies $f(-x,y,z)=f(x,-y,-z)=f(x,y,z)$, and $A(f)$ satisfies the defining condition for $\Omega_3$, then one can prove that $f(x,y,z)=f(x,-y,z)$, so one does not get anything more than in the $H_3(\GL_3(\Z);\Psi_{(a,b,c)})$ calculation.
\end{remark}
\begin{proof}
We mimic calculations in \cite{allison-ash-conrad}. We note that $\Omega_3(\sym(V))_{2n}$ is equal to $[\ext^3\sym(V)]_{2n}$ modulo the relation
\begin{equation}a b_{(1)}\wedge b_{(2)}\wedge c+a_{(1)}\wedge a_{(2)}b\wedge c=a\wedge S(b_{(1)})\wedge S(b_{(2)})c+
a\wedge S(b)c_{(1)}\wedge c_{(2)}. \tag{*}
\end{equation}
Dually, $\Omega_3(\sym(V))^*_e$ is the space of functionals $f\colon \sym(V)^{\otimes 3}\to \F$ satisfying
$f(x,y,z)=-f(y,x,z)=-f(x,z,y)$, $f(-x,-y,-z)=f(x,y,z)$ and $$f(x+y,y,z)+f(x,x+y,z)=f(x,-y,z-y)+f(x,-y+z,z).$$

Let $M=(\sym(V)^{\otimes 3}_{2n})^*$. Suppose
\begin{enumerate}
\item $f\in M^{U_P} $, 
\item $f(x,y,z)=-f(y,x,z)$
\item $f(x,y,z)=f(-x,-y,z)=f(x,y,-z)$
\item $f\cdot(\id+\gamma+\gamma^2)=0$, where $\gamma=\begin{bmatrix}0&1\\-1&-1\end{bmatrix}.$
\end{enumerate}
Then we claim that $A(f)\in\Omega_3(\sym(V))^*_e$.

Supposing this for the moment, 
$M\cong\oplus_\lambda \SF{\lambda}(V)\otimes \Psi^*_\lambda$, and $\Psi^*_\lambda=\Psi_{(2k-c,2k-b,2k-a)}$
and taking $U_P$-invariants yields
$M^{U_P}=\oplus_\lambda \SF{\lambda}(V)\otimes \Psi_{(2k-c,2k-b)}\cong \oplus_\lambda \SF{\lambda}(V)\otimes \Psi_{(b,c)}$. Modding out by $\id+\gamma+\gamma^2$ and $1+\tau$ yields 
$\oplus_{\lambda} \SF{\lambda}(V)\otimes \mathcal M^{\Omega}_{b-c}$, by the argument of Proposition  \ref{martin:prop}.
The other summands of the theorem follow from the calculation in Corollary \ref{cor:hom3}, where we need to show that there is no overlap except in the one dimensional subspace spanned by $E_{a,b,c}$. This follows because $A(M^{U_P})\cap A(M^{U_Q})$ is spanned by the $E_{a,b,c}$ terms. It suffices to show that the $D_{a,b,c}$ term is independent of $M^{U_Q}$ and $M^{U_P}$. This can be done by verifying that the only elements of $A(M^{U_P}+M^{U_Q})$ satisfying the $H_3(\GL_3(\Z);M)$ relations are the ones described in \cite{allison-ash-conrad}.

Now we prove the claim.
\begin{align*}
(Af)(x+y,y,z)&= f(x+y,y,z)-f(z,y,x+y)-f(x+y,z,y)\\
(Af)(x,x+y,z)&= f(x,x+y,z)-f(z,x+y,x)-f(x,z,x+y)\\
(Af)(x,-y,z-y)&=f(x,-y,z-y)-f(z-y,-y,x)-f(x,z-y,-y)\\
(Af)(x,-y+z,z)&=f(x,-y+z,z)-f(z,-y+z,x)-f(x,z,-y+z)
\end{align*}
Now, the the hypotheses on $f$ imply that $f(x+y,y,z)+f(x,x+y,z)=f(x,-y,z)$. 
So
\begin{align*}
(Af)(x+y,y,z)+&(Af)(x,x+y,z)=\\
&=f(x,-y,z)-f(z,y,x+y)-f(x+y,z,y)-f(z,x+y,x)-f(x,z,x+y)\\
&=f(x,-y,z)-f(z,y,x)-f(x+y,z,-x)-f(z,x+y,x)-f(x,z,y)\\
&=f(x,-y,z)-f(z,y,x)-f(x+y,z,x)-f(z,x+y,x)-f(x,z,y)\\
&=f(x,-y,z)-f(z,y,x)-f(x,z,y)
\end{align*}
where we use the fact that $f(a,b,c)=f(a,b,\lambda_1 a+\lambda_2b+c)$.
Similarly
\begin{align*}
(Af)&(x,-y,z-y)+(Af)(x,-y+z,z)=\\
&=f(x,-y,z-y)-f(z-y,-y,x)-f(x,z-y,-y)+f(x,-y+z,z)\\
&\,\,\,\,\,\,\,\,\,-f(z,-y+z,x)-f(x,z,-y+z)\\
&=-f(z,y,x)+f(x,-y,z-y)-f(x,z-y,-y)+f(x,-y+z,z)-f(x,z,-y+z)\\
&=-f(z,y,x)+f(x,-y,z)-f(x,z-y,z)+f(x,-y+z,z)-f(x,z,-y)\\
&=-f(z,y,x)+f(x,-y,z)-f(x,z,y)
\end{align*}
Thus both sides are equal.

%Now do the same thing for $U_Q$ invariants! That, is we assume that $f(x,y,z)=f(x,y+\lambda_1x,z)=f(x,y,z+\lambda_2 x)$ and
%that $f(x,y,z)=-f(x,z,y)=f(x,-y,-z)=f(-x,y,z)$ and that $f(x,y+z,z)+f(x,y,y+z)=f(x,y,-z)$.
%Now
%\begin{align*}
%(Af)(x+y,y,z)+(Af)(x,x+y,z)&=f(x+y,y,z)-f(y,x+y,z)-f(z,y,x+y)+f(x,x+y,z)-f(x+y,x,z)-f(z,x+y,x)\\
%&=f(x+y,y,z)-f(y,x+y,z)+f(x,x+y,z)-f(x+y,x,z) -f(z,x,-y)\\
%&=f(x+y,-x,z)-f(y,x,z)+f(x,y,z)-f(x+y,x,z)-f(z,x,-y)
%\end{align*}
%and
%\begin{align*}
%(Af)(x,-y,z-y)+(Af)(x,-y+z,z)&= f(x,-y,z-y)-f(-y,x,z-y)-f(z-y,-y,x)+f(x,-y+z,z)-f(-y+z,x,z)-f(z,-y+z,x)\\
%&=-f(-y,x,z-y)-f(z-y,-y,x)-f(-y+z,x,z)-f(z,-y+z,x)+f(x,-y,-z)\\
%&=-f(-y,x,z)-f(z-y,-z,x)-f(-y+z,x,z)-f(z,-y,x)+f(x,-y,-z)\\
%&=-f(y,x,z)+f(z-y,x,-z)-f(z-y,x,z)-f(z,y,-x)+f(x,y,z)
%\end{align*}
%\red{Sign convention alert!}

\end{proof}

\section{Proof of Theorems~\ref{thm:rank3pres}, ~\ref{thm:oddpres}, ~\ref{thm:omega3pres} }\label{sec:rank3proofs}
\begin{definition}
A graph in $\overline{\cG_{{H\Lie},1}}$ is said to be a ``chord diagram," if it is represented by a linear tree with ends joined by an edge (called the long chord). For example, 
\begin{minipage}{2cm}
\resizebox{2cm}{!}{\begin{tikzpicture}[scale=.8]
\coordinate(a) at (0,0);
\coordinate(b) at (1.5,0);
\coordinate(c) at (3.8,0);
\coordinate(d) at (5,0);
\coordinate(e) at (2,-1);
\coordinate(f) at (5,-.25);
\coordinate(g) at (1,1.5);
\coordinate(h) at (2.5,1.0);
\coordinate(i) at (3.8,.25);
\coordinate(j) at (5,.25);
\node[hopf, at=(g)] (G){${a}$};
\node[hopf, at=(h)] (H){${b}$};
\node[hopf, at=(e)] (E){${c}$};
\draw[thick] (a) to (d);
\draw[thick] (a) to[out=90, in=180] (G);
\draw[thick] (b) to[out=90, in=180] (H);
\draw[thick] (a) to[out=270, in =180] (E);
\draw[thick] (c) to (i);
\draw[thick] (d) to (j);
\draw[thick] (d) to (f);
\begin{scope}[decoration={markings,mark = at position 0.5 with {\arrow{stealth}}}]
\draw[densely dashed, postaction=decorate] (E) to[densely dashed, postaction=decorate, out=0, in=270](f);
\draw[densely dashed, postaction=decorate] (G) to[densely dashed, postaction=decorate,out=0,in=90](i);
\draw[densely dashed, postaction=decorate] (H) to[densely dashed, postaction=decorate,out=0,in=90](j);
\end{scope}
\end{tikzpicture}}
\end{minipage}
is a chord diagram. 
\end{definition}
By slide relations, we can assume that the hopf algebra elements at the ends of the tree are all at the beginnings of their respective edges, as in this picture. Note that a chord diagram can be rewritten as a sum of other chord diagrams in the following way. Choose an edge from among the top edges to become the new long chord. Now its two ends are joined by a subarc of the tree. Using IHX relations, lengthen this arc until it occupies the whole tree. This gives a linear combination of new chord diagrams. The equality of the original diagram with these new diagrams is called a ``chord shift relation." Some examples will be seen below. Chord diagrams also have a $\mathbb Z_2$ symmetry, which multiplies by a sign for each edge and applies the antipode to each element of $H$.
\begin{lemma}
$\overline{\cG_{{H\Lie},1}}$ is presented as a vector space by chord diagrams modulo chord shifting relations and the $\Z_2$-symmetry.
\end{lemma}
\begin{proof}
We define a map from $\overline{\cG_{{H\Lie},1}}$ to chord diagrams modulo the given relations.  Given an arbitrary graph in $\overline{\cG_{{H\Lie},1}}$, choose an edge to be the long chord. Then use IHX relations to expand as a sum of chord diagrams with this chord as the long one. We need to show that this does not depend on the choice of chord. Given a different choice, the two sets of diagrams are related to each other by chord shifting relations, so the map is well-defined. 

The map in the other direction is induced by inclusion, and is well-defined since the chord shift relations are IHX consequences. The composition of these two maps is the identity in one direction, implying that chord diagrams modulo chord shift relations inject into $\overline{\cG_{{H\Lie},1}}$. But because chord diagrams  generate $\overline{\cG_{{H\Lie},1}}$, this means that the two spaces are isomorphic.
\end{proof}

\begin{lemma}
$\partial(\overline{\cG_{{H\Lie},2}})$ is generated by boundaries of graphs of the following form. First one can assume both trees are linear (including the case of a tripod). If there are edges connecting a tree to itself, we may assume that at least one of them connects the ends of this tree. Finally, if there is no such self-edge, we may assume that tree is a tripod.
\end{lemma}
\begin{proof}
One can straighten out the trees and have edges join their ends using IHX relations. So the only thing that we need to check is that boundaries of trees without self-edges such that the trees are not tripods do not add anything to $\partial(\overline{\cG_{{H\Lie},2}})$. We proceed by induction on the number of leaves of the tree, the base case being a tripod with three leaves. Otherwise, one can choose an internal edge and break the tree into two pieces joined by a dashed edge. Form the graph which is the sum of contracting along each of the dashed edges emanating from the original tree (but not the new dashed edge we just created). This is still an element of  $\overline{\cG_{{H\Lie},2}}$. Taking its boundary, if an edge emanating from the left tree had been contracted, we can contract along an edge in the right tree while taking the boundary, and vice versa. These terms all cancel in pairs, leaving us with the single term where the dashed edge breaking apart the original tree is filled in. In this way, we have shown the boundary with the original tree is representable by boundaries of terms with trees of smaller size.
\end{proof}

Specifying now to the rank 3 case, the three generators of $\overline{\cG^{(3)}_{H\Lie,1}}$ are given below. 
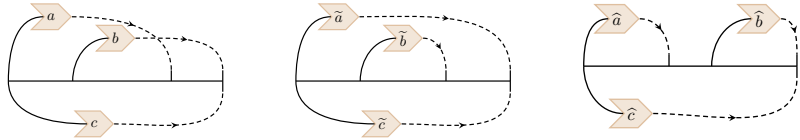
\begin{center}
\tp{3cm}{
\coordinate(a) at (0,0);
\coordinate(b) at (1.5,0);
\coordinate(c) at (3.8,0);
\coordinate(d) at (5,0);
\coordinate(e) at (2,-1);
\coordinate(f) at (5,-.25);
\coordinate(g) at (1,1.5);
\coordinate(h) at (2.5,1.0);
\coordinate(i) at (3.8,.25);
\coordinate(j) at (5,.25);
\node[hopf, at=(g)] (G){${a}$};
\node[hopf, at=(h)] (H){${b}$};
\node[hopf, at=(e)] (E){${c}$};
\draw[thick] (a) to (d);
\draw[thick] (a) to[out=90, in=180] (G);
\draw[thick] (b) to[out=90, in=180] (H);
\draw[thick] (a) to[out=270, in =180] (E);
\draw[thick] (c) to (i);
\draw[thick] (d) to (j);
\draw[thick] (d) to (f);
\begin{scope}[decoration={markings,mark = at position 0.5 with {\arrow{stealth}}}]
\draw[densely dashed, postaction=decorate] (E) to[densely dashed, postaction=decorate, out=0, in=270](f);
\draw[densely dashed, postaction=decorate] (G) to[densely dashed, postaction=decorate,out=0,in=90](i);
\draw[densely dashed, postaction=decorate] (H) to[densely dashed, postaction=decorate,out=0,in=90](j);
\end{scope}
}
\hspace{1em}
\tp{3cm}{
\coordinate(a) at (0,0);
\coordinate(b) at (1.5,0);
\coordinate(c) at (3.5,0);
\coordinate(d) at (5,0);
\coordinate(e) at (2,-1);
\coordinate(f) at (5,-.25);
\coordinate(g) at (1,1.5);
\coordinate(h) at (2.5,1.0);
\coordinate(i) at (3.5,.25);
\coordinate(j) at (5,.25);
\node[hopf, at=(g)] (G){$\wt{a}$};
\node[hopf, at=(h)] (H){$\wt{b}$};
\node[hopf, at=(e)] (E){$\wt{c}$};
\draw[thick] (a) to (d);
\draw[thick] (a) to[out=90, in=180] (G);
\draw[thick] (b) to[out=90, in=180] (H);
\draw[thick] (a) to[out=270, in =180] (E);
\draw[thick] (c) to (i);
\draw[thick] (d) to (j);
\draw[thick] (d) to (f);
\begin{scope}[decoration={markings,mark = at position 0.5 with {\arrow{stealth}}}]
\draw[densely dashed, postaction=decorate] (E) to[densely dashed, postaction=decorate, out=0, in=270](f);
\draw[densely dashed, postaction=decorate] (G) to[densely dashed, postaction=decorate,out=0,in=90](j);
\draw[densely dashed, postaction=decorate] (H) to[densely dashed, postaction=decorate,out=0,in=90](i);
\end{scope}
}
\hspace{1em}
\tp{3cm}{
\coordinate(a) at (0,0);
\coordinate(b) at (1.8,0);
\coordinate(c) at (2.7,0);
\coordinate(d) at (4.5,0);
\coordinate(e) at (1,-1);
\coordinate(f) at (4.5,-.25);
\coordinate(g) at (.7,1);
\coordinate(h) at (3.7,1);
\coordinate(i) at (1.8,.25);
\coordinate(j) at (4.5,.25);
\node[hopf, at=(g)] (G){$\wh{a}$};
\node[hopf, at=(h)] (H){$\wh{b}$};
\node[hopf, at=(e)] (E){$\wh{c}$};
\draw[thick] (a) to (d);
\draw[thick] (a) to[out=90, in=180] (G);
\draw[thick] (c) to[out=90, in=180] (H);
\draw[thick] (a) to[out=270, in =180] (E);
\draw[thick] (b) to (i);
\draw[thick] (d) to (j);
\draw[thick] (d) to (f);
\begin{scope}[decoration={markings,mark = at position 0.5 with {\arrow{stealth}}}]
\draw[densely dashed, postaction=decorate] (E) to[densely dashed, postaction=decorate, out=0, in=270](f);
\draw[densely dashed, postaction=decorate] (G) to[densely dashed, postaction=decorate,out=0,in=90](i);
\draw[densely dashed, postaction=decorate] (H) to[densely dashed, postaction=decorate,out=0,in=90](j);
\end{scope}
}
\end{center}

We identify them with elements of $\overline{H^{\otimes 3}}$. So the first generator corresponds to $a\otimes b\otimes c\in H^{\otimes 3}$. In order to distinquish the three graphs, we consider three copies of $\overline{H^{\otimes 3}}$, with elements of the second copy denoted $\widetilde{a}\otimes\widetilde{b}\otimes\widetilde{c}$ and elements of the third copy denoted $\wh{a}\otimes\wh{b}\otimes\wh{c}$. Depicting these three types of tensors pictorially, we use the symbols $\tp{1.5cm}{\boxtensor\transF{1}},\tp{1.5cm}{\roundtensor\transF{1}},$ and $\tp{1.5cm}{\tritensor\transF{1}}$.

We proceed to list the relations. There are symmetries of the graphs, chord shift relations, and relations coming from $\im\partial$.

\subsection{Symmetry relations}
The symmetries of the graphs correspond to the following relations: (One is actually a chord shift relation in disguise.)
\begin{center}
\begin{tabular}{|c|l|}
\hline
(S1)&$a\otimes b\otimes c=-S(b)\otimes S(a)\otimes S(c)$\\
\hline
(S2)& $\wt{a}\otimes \wt{b}\otimes \wt{c}=\wt{c}\otimes \wt{b}\otimes \wt{a}$\\
\hline
(S3)& $\wt{a}\otimes \wt{b}\otimes \wt{c}=-S(\wt{a})\otimes S(\wt{b})\otimes S(\wt{c})$\\
\hline
(S4)& $\wh{a}\otimes \wh{b}\otimes \wh{c}=-S(\wh{b})\otimes S(\wh{a})\otimes S(\wh{c})$\\
\hline
\end{tabular}
\end{center}

\subsection{Chord shift relations}
There are three chord shift relations as in the following chart, and proven below.
\begin{center}
\begin{tabular}{|c|l|}
\hline
(CS1)& $\wt{a}\otimes\wt{b}\otimes\wt{c}=b\otimes c\otimes a+b\otimes a\otimes c$\\
\hline
(CS2)& $ \wt{a}\otimes \wt{b}\otimes \wt{c}=\wt{c}\otimes \wt{a}\otimes \wt{b}-a\otimes c\otimes b-c\otimes a\otimes b+\wt{a}\otimes \wt{c}\otimes \wt{b}$\\
\hline
(CS3)& $ \wh{a}\otimes \wh{b}\otimes \wh{c}=\wt{c}\otimes \wt{a}\otimes \wt{b}-a\otimes c\otimes b-S(\wh{a})\otimes \wh{c}\otimes \wh{b}+S(a)\otimes c\otimes b$\\
\hline
\end{tabular}
\end{center}

 For the first type of chord diagram, there is only one other chord up to symmetry. So without loss of generality, we choose the $b$ chord to expand into the long chord. 
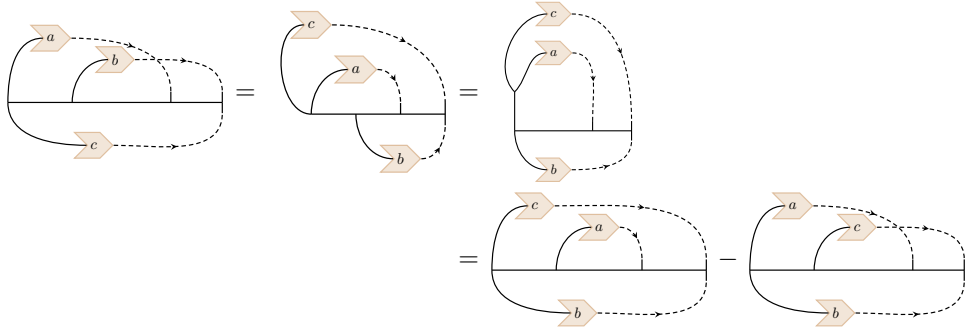
\begin{align*}
\tp{3cm}{
\coordinate(a) at (0,0);
\coordinate(b) at (1.5,0);
\coordinate(c) at (3.8,0);
\coordinate(d) at (5,0);
\coordinate(e) at (2,-1);
\coordinate(f) at (5,-.25);
\coordinate(g) at (1,1.5);
\coordinate(h) at (2.5,1.0);
\coordinate(i) at (3.8,.25);
\coordinate(j) at (5,.25);
\node[hopf, at=(g)] (G){${a}$};
\node[hopf, at=(h)] (H){${b}$};
\node[hopf, at=(e)] (E){${c}$};
\draw[thick] (a) to (d);
\draw[thick] (a) to[out=90, in=180] (G);
\draw[thick] (b) to[out=90, in=180] (H);
\draw[thick] (a) to[out=270, in =180] (E);
\draw[thick] (c) to (i);
\draw[thick] (d) to (j);
\draw[thick] (d) to (f);
\begin{scope}[decoration={markings,mark = at position 0.5 with {\arrow{stealth}}}]
\draw[densely dashed, postaction=decorate] (E) to[densely dashed, postaction=decorate, out=0, in=270](f);
\draw[densely dashed, postaction=decorate] (G) to[densely dashed, postaction=decorate,out=0,in=90](i);
\draw[densely dashed, postaction=decorate] (H) to[densely dashed, postaction=decorate,out=0,in=90](j);
\end{scope}
}
=
\tp{2.5cm}{
\coordinate(a) at (0,0);
\coordinate(b) at (1,0);
\coordinate(c) at (2,0);
\coordinate(d) at (3,0);
\coordinate (i) at (0,2);
\coordinate(h) at (1,1);
\coordinate(j) at (2,.25);
\coordinate(k) at (3,.25);
\coordinate(e) at (1,-.25);
\coordinate(f) at (3,-.25);
\coordinate(g) at (2,-1);
\node[hopf, at=(i)](I){$c$};
\node[hopf, at=(h)](H){$a$};
\node[hopf, at=(g)](G){$b$};
\draw[thick] (a) to (d);
\draw[thick] (a) to[out=180,in=180] (I);
\draw[thick] (a) to[out=90,in=180] (H);
\draw[thick] (b) to[out=270,in=180] (G);
\draw[thick] (c) to (j);
\draw[thick] (f) to (k);
  \begin{scope}[decoration={markings,mark = at position 0.5 with {\arrow{stealth}}}]
\draw[densely dashed, postaction=decorate] (I) to[densely dashed, postaction=decorate, out=0, in=90](k);
\draw[densely dashed, postaction=decorate] (H) to[densely dashed, postaction=decorate,out=0,in=90](j);
\draw[densely dashed, postaction=decorate] (G) to[densely dashed, postaction=decorate,out=0,in=270](f);
\end{scope}
}
&=
\tp{2cm}{
\coordinate(a) at (0,0);
\coordinate(d) at (0,1);
\coordinate(f) at (1,2);
\coordinate(e) at (1,3);
\coordinate(g) at (1,-1);
\coordinate(b) at (2,0);
\coordinate(c) at (3,0);
\coordinate(i) at (2,.25);
\coordinate(j) at (3,.25);
\coordinate(h) at (3,-.25);
\node[hopf, at=(e)](E){$c$};
\node[hopf, at=(f)](F){$a$};
\node[hopf, at=(g)](G){$b$};
\draw[thick] (a) to (d);
\draw[thick] (a) to (c);
\draw[thick](a) to[out =270, in =180] (G);
\draw[thick](d) to[out=135,in=180] (E);
\draw[thick] (d) to[out=45, in =180] (F);
\draw[thick] (b) to (i);
\draw[thick] (h) to (j);
  \begin{scope}[decoration={markings,mark = at position 0.5 with {\arrow{stealth}}}]
\draw[densely dashed, postaction=decorate] (E) to[densely dashed, postaction=decorate, out=0, in=90](j);
\draw[densely dashed, postaction=decorate] (F) to[densely dashed, postaction=decorate,out=0,in=90](i);
\draw[densely dashed, postaction=decorate] (G) to[densely dashed, postaction=decorate,out=0,in=270](h);
\end{scope}
}
\\&=
\tp{3cm}{
\coordinate(a) at (0,0);
\coordinate(b) at (1.5,0);
\coordinate(c) at (3.5,0);
\coordinate(d) at (5,0);
\coordinate(e) at (2,-1);
\coordinate(f) at (5,-.25);
\coordinate(g) at (1,1.5);
\coordinate(h) at (2.5,1.0);
\coordinate(i) at (3.5,.25);
\coordinate(j) at (5,.25);
\node[hopf, at=(g)] (G){$c$};
\node[hopf, at=(h)] (H){$a$};
\node[hopf, at=(e)] (E){$b$};
\draw[thick] (a) to (d);
\draw[thick] (a) to[out=90, in=180] (G);
\draw[thick] (b) to[out=90, in=180] (H);
\draw[thick] (a) to[out=270, in =180] (E);
\draw[thick] (c) to (i);
\draw[thick] (d) to (j);
\draw[thick] (d) to (f);
\begin{scope}[decoration={markings,mark = at position 0.5 with {\arrow{stealth}}}]
\draw[densely dashed, postaction=decorate] (E) to[densely dashed, postaction=decorate, out=0, in=270](f);
\draw[densely dashed, postaction=decorate] (G) to[densely dashed, postaction=decorate,out=0,in=90](j);
\draw[densely dashed, postaction=decorate] (H) to[densely dashed, postaction=decorate,out=0,in=90](i);
\end{scope}
}
-
\tp{3cm}{
\coordinate(a) at (0,0);
\coordinate(b) at (1.5,0);
\coordinate(c) at (3.8,0);
\coordinate(d) at (5,0);
\coordinate(e) at (2,-1);
\coordinate(f) at (5,-.25);
\coordinate(g) at (1,1.5);
\coordinate(h) at (2.5,1.0);
\coordinate(i) at (3.8,.25);
\coordinate(j) at (5,.25);
\node[hopf, at=(g)] (G){${a}$};
\node[hopf, at=(h)] (H){${c}$};
\node[hopf, at=(e)] (E){${b}$};
\draw[thick] (a) to (d);
\draw[thick] (a) to[out=90, in=180] (G);
\draw[thick] (b) to[out=90, in=180] (H);
\draw[thick] (a) to[out=270, in =180] (E);
\draw[thick] (c) to (i);
\draw[thick] (d) to (j);
\draw[thick] (d) to (f);
\begin{scope}[decoration={markings,mark = at position 0.5 with {\arrow{stealth}}}]
\draw[densely dashed, postaction=decorate] (E) to[densely dashed, postaction=decorate, out=0, in=270](f);
\draw[densely dashed, postaction=decorate] (G) to[densely dashed, postaction=decorate,out=0,in=90](i);
\draw[densely dashed, postaction=decorate] (H) to[densely dashed, postaction=decorate,out=0,in=90](j);
\end{scope}
}
\end{align*}

The relation becomes $$a\otimes b\otimes c=\wt{c}\otimes \wt{a}\otimes \wt{b}-a\otimes c\otimes b,$$
and by renaming this is equivalent to $$(CS1)\,\, \wt{a}\otimes\wt{b}\otimes\wt{c}=b\otimes c\otimes a+b\otimes a\otimes c.$$ 
Again for the second generator, there is only one other chord to consider.
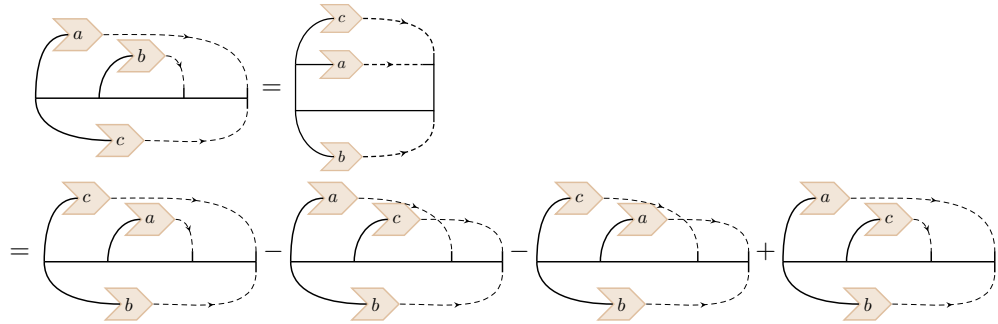
\begin{multline*}
\graphB{a}{b}{c}=
\tp{2cm}{
\coordinate(a) at (0,0);
\coordinate(b) at (1,-1);
\coordinate(c) at (3,-.25);
\coordinate(d) at (3,0);
\coordinate(e) at (3,1);
\coordinate(f) at (3,1.25);
\coordinate(g) at (2.75,1);
\coordinate(h) at (1,1);
\coordinate(i) at (0,1);
\coordinate(j) at (1,2);
\node[hopf, at=(j)](J){$c$};
\node[hopf, at=(h)](H){$a$};
\node[hopf, at=(b)](B){$b$};
\draw[thick] (i) to (a) to (d) to (e);
\draw[thick] (c) to (d);
\draw[thick] (g) to (e) to (f);
\draw[thick] (i) to[out=90, in =180] (J);
\draw[thick](i) to[out=0,in=180] (H);
\draw[thick](a) to[out=270, in=180](B);
\begin{scope}[decoration={markings,mark = at position 0.5 with {\arrow{stealth}}}]
\draw[densely dashed, postaction=decorate] (J) to[densely dashed, postaction=decorate, out=0, in=90](f);
\draw[densely dashed, postaction=decorate] (H) to[densely dashed, postaction=decorate,out=0,in=180](g);
\draw[densely dashed, postaction=decorate] (B) to[densely dashed, postaction=decorate,out=0,in=270](c);
\end{scope}
}\\
=\graphB{c}{a}{b}-\graphA{a}{c}{b}
-\graphA{c}{a}{b}+\graphB{a}{c}{b}
\end{multline*}
  yielding
$$(CS2)\,\, \wt{a}\otimes \wt{b}\otimes \wt{c}=\wt{c}\otimes \wt{a}\otimes \wt{b}-a\otimes c\otimes b-c\otimes a\otimes b+\wt{a}\otimes \wt{c}\otimes \wt{b}$$
The third relation is:
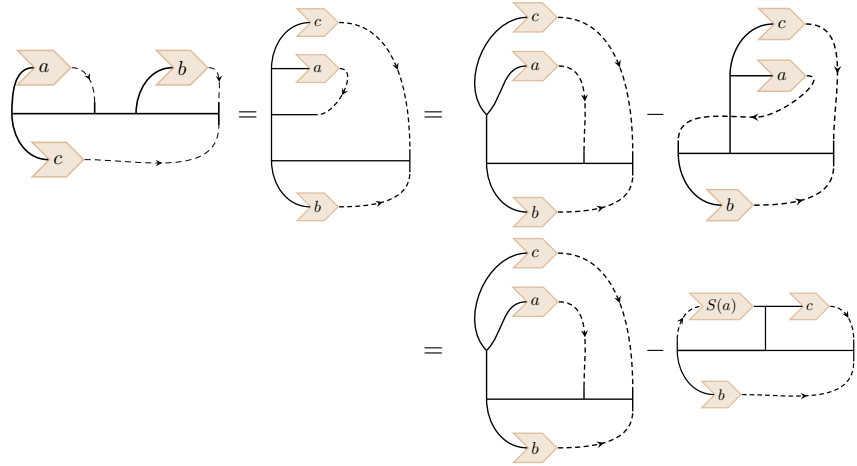
\begin{align*}\graphC{a}{b}{c}=
\tp{2cm}{
\coordinate(a) at (0,0);
\coordinate(b) at (1,-1);
\coordinate(c) at (3,-.25);
\coordinate(d) at(3,0);
\coordinate(e) at(3,.25);
\coordinate(f) at (1,3);
\coordinate(g)at (0,2);
\coordinate(h) at(0,1);
\coordinate(i) at(1,1);
\coordinate(j) at(1,2);
\node[hopf, at=(f)](F){$c$};
\node[hopf,at=(j)](J){$a$};
\node[hopf,at=(b)](B){$b$};
\draw[thick] (d) to (a) to (g);
\draw[thick] (c) to (e);
\draw[thick] (h) to (i);
\draw[thick] (g) to[out=90, in=180] (F);
\draw[thick] (g) to[out=0, in=180] (J);
\draw[thick] (a) to[out=270, in=180] (B);
\begin{scope}[decoration={markings,mark = at position 0.5 with {\arrow{stealth}}}]
\draw[densely dashed, postaction=decorate] (F) to[densely dashed, postaction=decorate, out=0, in=90](e);
\draw[densely dashed, postaction=decorate] (J) to[densely dashed, postaction=decorate,out=0,in=0](i);
\draw[densely dashed, postaction=decorate] (B) to[densely dashed, postaction=decorate,out=0,in=270](c);
\end{scope}
}&=
\tp{2.5cm}{
\coordinate(a) at (0,0);
\coordinate(d) at (0,1);
\coordinate(f) at (1,2);
\coordinate(e) at (1,3);
\coordinate(g) at (1,-1);
\coordinate(b) at (2,0);
\coordinate(c) at (3,0);
\coordinate(i) at (2,.25);
\coordinate(j) at (3,.25);
\coordinate(h) at (3,-.25);
\node[hopf, at=(e)](E){$c$};
\node[hopf, at=(f)](F){$a$};
\node[hopf, at=(g)](G){$b$};
\draw[thick] (a) to (d);
\draw[thick] (a) to (c);
\draw[thick](a) to[out =270, in =180] (G);
\draw[thick](d) to[out=135,in=180] (E);
\draw[thick] (d) to[out=45, in =180] (F);
\draw[thick] (b) to (i);
\draw[thick] (h) to (j);
  \begin{scope}[decoration={markings,mark = at position 0.5 with {\arrow{stealth}}}]
\draw[densely dashed, postaction=decorate] (E) to[densely dashed, postaction=decorate, out=0, in=90](j);
\draw[densely dashed, postaction=decorate] (F) to[densely dashed, postaction=decorate,out=0,in=90](i);
\draw[densely dashed, postaction=decorate] (G) to[densely dashed, postaction=decorate,out=0,in=270](h);
\end{scope}
}-
\tp{2.5cm}{
\coordinate(a) at (0,0);
\coordinate(b) at (1,-1);
\coordinate(c) at (3,-.25);
\coordinate(d) at (3,0);
\coordinate(e) at (3,.25);
\coordinate(f) at (2,2.5);
\coordinate(g) at (1,1.5);
\coordinate(h) at (2,1.5);
\coordinate(i) at (0,.25);
\coordinate(j) at (1,0);
\coordinate(k) at (1.5,.7);
\node[hopf, at=(f)](F){$c$};
\node[hopf, at=(h)](H){$a$};
\node[hopf, at=(b)](B){$b$};
\draw[thick] (a) to (d);
\draw[thick] (a) to (i);
\draw[thick] (a) to[out=270,in=180] (B);
\draw[thick] (j) to (g);
\draw[thick] (g) to (H);
\draw[thick](g) to[out=90,in=180] (F);
\draw (c) to (e);
 \begin{scope}[decoration={markings,mark = at position 0.5 with {\arrow{stealth}}}]
\draw[densely dashed, postaction=decorate] (F) to[densely dashed, postaction=decorate, out=0, in=90](e);
\draw[densely dashed, postaction=decorate] (H) to[densely dashed, postaction=decorate,out=0,in=0](k) to[densely dashed, out=180,in=90] (i);
\draw[densely dashed, postaction=decorate] (B) to[densely dashed, postaction=decorate,out=0,in=270](c);
\end{scope}
}\\
&=\tp{2.5cm}{
\coordinate(a) at (0,0);
\coordinate(d) at (0,1);
\coordinate(f) at (1,2);
\coordinate(e) at (1,3);
\coordinate(g) at (1,-1);
\coordinate(b) at (2,0);
\coordinate(c) at (3,0);
\coordinate(i) at (2,.25);
\coordinate(j) at (3,.25);
\coordinate(h) at (3,-.25);
\node[hopf, at=(e)](E){$c$};
\node[hopf, at=(f)](F){$a$};
\node[hopf, at=(g)](G){$b$};
\draw[thick] (a) to (d);
\draw[thick] (a) to (c);
\draw[thick](a) to[out =270, in =180] (G);
\draw[thick](d) to[out=135,in=180] (E);
\draw[thick] (d) to[out=45, in =180] (F);
\draw[thick] (b) to (i);
\draw[thick] (h) to (j);
  \begin{scope}[decoration={markings,mark = at position 0.5 with {\arrow{stealth}}}]
\draw[densely dashed, postaction=decorate] (E) to[densely dashed, postaction=decorate, out=0, in=90](j);
\draw[densely dashed, postaction=decorate] (F) to[densely dashed, postaction=decorate,out=0,in=90](i);
\draw[densely dashed, postaction=decorate] (G) to[densely dashed, postaction=decorate,out=0,in=270](h);
\end{scope}
}-
\tp{2.5cm}{
\coordinate(a) at (0,0);
\coordinate(b) at (1,-1);
\coordinate(c) at (4,-.25);
\coordinate(d) at (4,0);
\coordinate(e) at (4,.25);
\coordinate(f) at (3,1);
\coordinate(g) at (2,1);
\coordinate(h) at (1,1);
\coordinate(i) at (2,0);
\coordinate(j) at (0,.25);
\node[hopf,at=(h)](H){$S(a)$};
\node[hopf,at=(f)](F){$c$};
\node[hopf,at=(b)](B){$b$};
\draw[thick] (a) to (d);
\draw[thick](a) to (i);
\draw[thick](a) to (j);
\draw[thick] (H) to (F);
\draw[thick] (i) to (g);
\draw[thick] (c) to (e);
\draw[thick] (a) to[out=270,in=180] (B);
 \begin{scope}[decoration={markings,mark = at position 0.5 with {\arrow{stealth}}}]
\draw[densely dashed, postaction=decorate] (j) to[densely dashed, postaction=decorate, out=90, in=180](H);
\draw[densely dashed, postaction=decorate] (F) to[densely dashed, postaction=decorate,out=0,in=90](e);
\draw[densely dashed, postaction=decorate] (B) to[densely dashed, postaction=decorate,out=0,in=270](c);
\end{scope}
}
\end{align*}
Using IHX and slide relations, this becomes
$$\graphB{c}{a}{b}-\graphA{a}{c}{b}-\graphC{S(a)}{c}{b}+\graphA{S(a)}{c}{b}.$$
Yielding the third chord shift relation:
$$(CS3)\,\, \wh{a}\otimes \wh{b}\otimes \wh{c}=\wt{c}\otimes \wt{a}\otimes \wt{b}-a\otimes c\otimes b-S(\wh{a})\otimes \wh{c}\otimes \wh{b}+S(a)\otimes c\otimes b$$ 

\subsection{Boundary relations}
There are four boundary relations.
\begin{center}
\begin{tabular}{|c|l|}
\hline
(D1)& $S(\wt{a}_{(1)})\otimes S(\wt{a}_{(2)})\wt{b}\otimes\wt{c}-S(a_{(1)})b\otimes S(a_{(2)})\otimes c-S(\wh{a})\wh{b}_{(1)}\otimes S(\wh{b}_{(2)})\otimes\wh{c}$\\
&$+S(a)b_{(1)}\otimes S(b_{(2)})\otimes c+S(\wt{a})\otimes S(\wt{b})\otimes \wt{c}-S(a)\otimes S(b)\otimes c=0$ \\
\hline
(D2)& $\wh{a}\otimes \wh{b}\otimes \wh{c}=-S(\wh{a})\otimes S(\wh{b})\otimes S(\wh{c})$\\
\hline
(D3)&  $\wt{a}\otimes \wt{b}\otimes \wt{c}=-\wt{a}\otimes S(\wt{b})\otimes \wt{c}$\\
\hline
(D4)& $\wh{a}\otimes \wh{b}\otimes \wh{c}=-S(\wh{a})\otimes \wh{b}\otimes \wh{c}$\\
\hline
\end{tabular}
\end{center}

These are proven as follows.
\begin{align*}
\tp{3cm}{
\coordinate(a) at (0,0);
\coordinate(b) at (1,-1);
\coordinate(c) at (2,-.25);
\coordinate(d) at (2,0);
\coordinate(e) at (2,.25);
\coordinate(f) at (1,.25);
\coordinate(g) at (0,.25);
\coordinate(h) at (1,0);
\coordinate(i) at (-2,2);
\coordinate(j) at (-1,3);
\coordinate(k) at (-1,2);
\coordinate(l) at (-1,1);
\node[hopf, at=(b)](B){$c$};
\node[hopf, at=(k)](K){$b$};
\node[hopf, at=(l)](L){$a$};
\node[empty, at=(j)](J){};
\draw[thick] (g) to (a) to (d) to (e);
\draw[thick](f) to (h);
\draw[thick](c) to (d);
\draw[thick](a) to[out=270, in =180] (B);
\draw[thick] (i) to[out=90,in=180] (J);
\draw[thick] (i) to[out=0,in=180] (K);
\draw[thick] (i) to[out=-90,in=180] (L);
 \begin{scope}[decoration={markings,mark = at position 0.5 with {\arrow{stealth}}}]
\draw[densely dashed, postaction=decorate] (J) to[densely dashed, postaction=decorate, out=0, in=90](e);
\draw[densely dashed, postaction=decorate] (K) to[densely dashed, postaction=decorate,out=0,in=90](f);
\draw[densely dashed, postaction=decorate] (L) to[densely dashed, postaction=decorate,out=0,in=90](g);
\draw[densely dashed, postaction=decorate] (B) to[densely dashed, postaction=decorate,out=0,in=270](c);
\end{scope}
}
&\mapsto
\tp{3cm}{
\coordinate(a) at (0,0);
\coordinate(b) at (1,-1);
\coordinate(c) at (2,-.25);
\coordinate(d) at (2,0);
\coordinate(e) at (2,.25);
\coordinate(f) at (1,.25);
\coordinate(g) at (0,.25);
\coordinate(h) at (1,0);
\coordinate(i) at (-2,2);
\coordinate(j) at (-1,3);
\coordinate(k) at (-1,2);
\coordinate(l) at (-1,1);
\node[hopf, at=(b)](B){$c$};
\node[hopf, at=(k)](K){$b$};
\node[hopf, at=(l)](L){$a$};
\node[empty, at=(j)](J){};
\draw[thick] (g) to (a) to (d) to (e);
\draw[thick](f) to (h);
\draw[thick](c) to (d);
\draw[thick](a) to[out=270, in =180] (B);
\draw[thick] (i) to[out=90,in=180] (J);
\draw[thick] (i) to[out=0,in=180] (K);
\draw[thick] (i) to[out=-90,in=180] (L);
 \begin{scope}[decoration={markings,mark = at position 0.5 with {\arrow{stealth}}}]
\draw[thick] (J) to[out=0, in=90](e);
\draw [densely dashed, postaction=decorate](K) to[densely dashed, postaction=decorate,out=0,in=90](f);
\draw[densely dashed, postaction=decorate] (L) to[densely dashed, postaction=decorate,out=0,in=90](g);
\draw[densely dashed, postaction=decorate] (B) to[densely dashed, postaction=decorate,out=0,in=270](c);
\end{scope}
}+
\tp{3cm}{
\coordinate(a) at (0,0);
\coordinate(b) at (1,-1);
\coordinate(c) at (2,-.25);
\coordinate(d) at (2,0);
\coordinate(e) at (2,.25);
\coordinate(f) at (1,.25);
\coordinate(g) at (0,.25);
\coordinate(h) at (1,0);
\coordinate(i) at (-2.5,2);
\coordinate(j) at (-1,3);
\coordinate(k) at (-1,2);
\coordinate(l) at (-1,1);
\node[hopf, at=(b)](B){$c$};
\node[empty, at=(k)](K){};
\node[hopf, at=(l)](L){$S(b_{(2)})a$};
\node[hopf, at=(j)](J){$S(b_{(1)})$};
\draw[thick] (g) to (a) to (d) to (e);
\draw[thick](f) to (h);
\draw[thick](c) to (d);
\draw[thick](a) to[out=270, in =180] (B);
\draw[thick] (i) to[out=90,in=180] (J);
\draw[thick] (i) to[out=0,in=180] (K);
\draw[thick] (i) to[out=-90,in=180] (L);
 \begin{scope}[decoration={markings,mark = at position 0.5 with {\arrow{stealth}}}]
\draw[densely dashed, postaction=decorate] (J) to[densely dashed, postaction=decorate, out=0, in=90](e);
\draw (K) to[thick,out=0,in=90](f);
\draw[densely dashed, postaction=decorate] (L) to[densely dashed, postaction=decorate,out=0,in=90](g);
\draw[densely dashed, postaction=decorate] (B) to[densely dashed, postaction=decorate,out=0,in=270](c);
\end{scope}
}+
\tp{3cm}{
\coordinate(a) at (0,0);
\coordinate(b) at (1,-1);
\coordinate(c) at (2,-.25);
\coordinate(d) at (2,0);
\coordinate(e) at (2,.25);
\coordinate(f) at (1,.25);
\coordinate(g) at (0,.25);
\coordinate(h) at (1,0);
\coordinate(i) at (-2.5,2);
\coordinate(j) at (-1,3);
\coordinate(k) at (-1,2);
\coordinate(l) at (-1,1);
\node[hopf, at=(b)](B){$c$};
\node[hopf, at=(k)](K){$S(a_{(2)})b$};
\node[empty, at=(l)](L){};
\node[hopf, at=(j)](J){$S(a_{(1)})$};
\draw[thick] (g) to (a) to (d) to (e);
\draw[thick](f) to (h);
\draw[thick](c) to (d);
\draw[thick](a) to[out=270, in =180] (B);
\draw[thick] (i) to[out=90,in=180] (J);
\draw[thick] (i) to[out=0,in=180] (K);
\draw[thick] (i) to[out=-90,in=180] (L);
 \begin{scope}[decoration={markings,mark = at position 0.5 with {\arrow{stealth}}}]
\draw[densely dashed, postaction=decorate] (J) to[densely dashed, postaction=decorate, out=0, in=90](e);
\draw[densely dashed, postaction=decorate] (K) to[densely dashed, postaction=decorate,out=0,in=90](f);
\draw (L) to[thick,out=0,in=90](g);
\draw[densely dashed, postaction=decorate] (B) to[densely dashed, postaction=decorate,out=0,in=270](c);
\end{scope}
}
\\
&=
\tp{2.5cm}{
\coordinate(a) at (0,0);
\coordinate(b) at (.7,0);
\coordinate(c) at (3,0);
\coordinate(d) at (3,-.25);
\coordinate(e) at (1,-1);
\coordinate(f) at (0,.25);
\coordinate(g) at (.7,.25);
\coordinate(h) at (3,.25);
\coordinate(i) at (2,2);
\coordinate(j) at (2,1);
\coordinate(k) at (3,1.5);
\node[hopf, at=(i)](I){$S(a)$};
\node[hopf, at=(j)](J){$S(b)$};
\node[hopf, at=(e)](E){$c$};
\draw[thick] (f) to (a) to (c) to (d);
\draw[thick] (a) to[out=-90,in=180] (E);
\draw[thick] (b) to (g);
\draw[thick] (I) to[out=0,in=135] (k);
\draw[thick] (J) to[out=0,in=-135] (k);
\draw[thick] (k) to[out=0,in=90] (c);
 \begin{scope}[decoration={markings,mark = at position 0.5 with {\arrow{stealth}}}]
\draw[densely dashed, postaction=decorate] (f) to[densely dashed, postaction=decorate, out=90, in=180](I);
\draw[densely dashed, postaction=decorate] (g) to[densely dashed, postaction=decorate,out=90,in=180](J);
\draw[densely dashed, postaction=decorate] (E) to[densely dashed, postaction=decorate,out=0,in=270](d);
\end{scope}
}
-
\tp{3cm}{
\coordinate(a) at (0,0);
\coordinate(b) at (1,-1);
\coordinate(c) at (4,-.25);
\coordinate(d) at (4,0);
\coordinate(e) at (4,.25);
\coordinate(f) at (3,1);
\coordinate(g) at (2,1);
\coordinate(h) at (1,1);
\coordinate(i) at (2,0);
\coordinate(j) at (0,.25);
\node[hopf,at=(h)](H){$S(a)b_{(1)}$};
\node[hopf,at=(f)](F){$S(b_{(2)})$};
\node[hopf,at=(b)](B){$c$};
\draw[thick] (a) to (d);
\draw[thick](a) to (i);
\draw[thick](a) to (j);
\draw[thick] (H) to (F);
\draw[thick] (i) to (g);
\draw[thick] (c) to (e);
\draw[thick] (a) to[out=270,in=180] (B);
 \begin{scope}[decoration={markings,mark = at position 0.5 with {\arrow{stealth}}}]
\draw[densely dashed, postaction=decorate] (j) to[densely dashed, postaction=decorate, out=90, in=180](H);
\draw[densely dashed, postaction=decorate] (F) to[densely dashed, postaction=decorate,out=0,in=90](e);
\draw[densely dashed, postaction=decorate] (B) to[densely dashed, postaction=decorate,out=0,in=270](c);
\end{scope}
}+\tp{2.5cm}{
\coordinate(a) at (0,0);
\coordinate(d) at (0,1);
\coordinate(f) at (1,2);
\coordinate(e) at (1,3);
\coordinate(g) at (1,-1);
\coordinate(b) at (2,0);
\coordinate(c) at (3,0);
\coordinate(i) at (2,.25);
\coordinate(j) at (3,.25);
\coordinate(h) at (3,-.25);
\node[hopf, at=(e)](E){$S(a_{(1)})$};
\node[hopf, at=(f)](F){$S(a_{(2)})b$};
\node[hopf, at=(g)](G){$c$};
\draw[thick] (a) to (d);
\draw[thick] (a) to (c);
\draw[thick](a) to[out =270, in =180] (G);
\draw[thick](d) to[out=135,in=180] (E);
\draw[thick] (d) to[out=45, in =180] (F);
\draw[thick] (b) to (i);
\draw[thick] (h) to (j);
  \begin{scope}[decoration={markings,mark = at position 0.5 with {\arrow{stealth}}}]
\draw[densely dashed, postaction=decorate] (E) to[densely dashed, postaction=decorate, out=0, in=90](j);
\draw[densely dashed, postaction=decorate] (F) to[densely dashed, postaction=decorate,out=0,in=90](i);
\draw[densely dashed, postaction=decorate] (G) to[densely dashed, postaction=decorate,out=0,in=270](h);
\end{scope}
}\\
&=\graphB{S(a)}{S(b)}{c}-\graphA{S(a)}{S(b)}{c}-\graphC{S(a)b_{(1)}}{S(b_{(2)})}{c}\\&+\graphA{S(a)b_{(1)}}{S(b_{(2)})}{c}+
\graphB{S(a_{(1)})}{S(a_{(2)})b}{c}-\graphA{S(a_{(2)})b}{S(a_{(1)})}{c}\\
&=S(\wt{a})\otimes S(\wt{b})\otimes\wt{c}-S(a)\otimes S(b)\otimes c-S(\wh{a})\wh{b}_{(1)}\otimes S(\wh{b}_{(2)})\otimes\wh{c}\\
&+S(a)b_{(1)}\otimes S( b_{(2)})\otimes c+
S(\wt{a}_{(1)})S(\wt{a}_{(2)}\wt{b})\otimes \wt{c}-S(a_{(1)})b\otimes S(a_{(2)})\otimes c
\end{align*}
This concludes the proof of (D1).

To prove (D2) we consider the following boundary.
\begin{multline*}
\begin{minipage}{5.5cm}
\resizebox{5.5cm}{!}{
\begin{tikzpicture}
\coordinate (a) at (0,0);
\coordinate (b) at(1,0);
\coordinate (c) at (2,0);
\coordinate (d) at (3.5,0);
\coordinate(e) at (4,0);
\coordinate (f) at (5,0);
\coordinate (g) at (5,-1);
\coordinate (h) at (4,-1);
\coordinate (i) at (3.5,-1);
\coordinate (j) at (1.5,-1);
\coordinate(k) at (1,-1);
\coordinate(l) at (0,-1);
\node[hopf, at=(a)](A){$a$};
\node[hopf, at=(c)](C){$c$};
\node[hopf, at=(f)](F){$b$};
\draw[thick] (A) to (b) to (k) to (l);
\draw[thick] (b) to (C);
\draw[thick] (k) to (j);
\draw[thick] (d) to (F);
\draw[thick] (i) to (g);
\draw[thick](e) to (h);
  \begin{scope}[decoration={markings,mark = at position 0.5 with {\arrow{stealth}}}]
\draw[densely dashed, postaction=decorate] (l) to[densely dashed, postaction=decorate, out=180, in=180](A);
\draw[densely dashed, postaction=decorate] (C) to[densely dashed, postaction=decorate,out=0,in=180](d);
\draw[densely dashed, postaction=decorate] (j) to[densely dashed, postaction=decorate,out=0,in=180](i);
\draw[densely dashed, postaction=decorate] (F) to[densely dashed, postaction=decorate,out=0,in=0](g);
\end{scope}
\end{tikzpicture}
}
\end{minipage}
\mapsto
\begin{minipage}{5.5cm}
\resizebox{5.5cm}{!}{
\begin{tikzpicture}
\coordinate (a) at (0,0);
\coordinate (b) at(1,0);
\coordinate (c) at (2,0);
\coordinate (d) at (3.5,0);
\coordinate(e) at (4,0);
\coordinate (f) at (5,0);
\coordinate (g) at (5,-1);
\coordinate (h) at (4,-1);
\coordinate (i) at (3.5,-1);
\coordinate (j) at (1.5,-1);
\coordinate(k) at (1,-1);
\coordinate(l) at (0,-1);
\node[hopf, at=(a)](A){$a$};
\node[hopf, at=(c)](C){$c$};
\node[hopf, at=(f)](F){$b$};
\draw[thick] (A) to (b) to (k) to (l);
\draw[thick] (b) to (C);
\draw[thick] (k) to (j);
\draw[thick] (d) to (F);
\draw[thick] (i) to (g);
\draw[thick](e) to (h);
  \begin{scope}[decoration={markings,mark = at position 0.5 with {\arrow{stealth}}}]
\draw[densely dashed, postaction=decorate] (l) to[densely dashed, postaction=decorate, out=180, in=180](A);
\draw[densely dashed, postaction=decorate] (C) to[densely dashed, postaction=decorate,out=0,in=180](d);
\draw[thick] (j) to[thick,out=0,in=180](i);
\draw[densely dashed, postaction=decorate] (F) to[densely dashed, postaction=decorate,out=0,in=0](g);
\end{scope}
\end{tikzpicture}
}
\end{minipage}
+\\
\begin{minipage}{5.5cm}
\resizebox{5.5cm}{!}{
\begin{tikzpicture}
\coordinate (a) at (0,0);
\coordinate (b) at(1,0);
\coordinate (c) at (2,0);
\coordinate (d) at (3.5,0);
\coordinate(e) at (4,0);
\coordinate (f) at (5,0);
\coordinate (g) at (5,-1);
\coordinate (h) at (4,-1);
\coordinate (i) at (3.5,-1);
\coordinate (j) at (2,-1);
\coordinate(k) at (1,-1);
\coordinate(l) at (0,-1);
\node[hopf, at=(a)](A){$a$};
\node[hopf, at=(j)](J){$S(c)$};
\node[hopf, at=(f)](F){$b$};
\draw[thick] (A) to (b) to (k) to (l);
\draw[thick] (b) to (c);
\draw[thick] (k) to (J);
\draw[thick] (d) to (F);
\draw[thick] (i) to (g);
\draw[thick](e) to (h);
  \begin{scope}[decoration={markings,mark = at position 0.5 with {\arrow{stealth}}}]
\draw[densely dashed, postaction=decorate] (l) to[densely dashed, postaction=decorate, out=180, in=180](A);
\draw[thick] (c) to[thick,out=0,in=180](d);
\draw[densely dashed, postaction=decorate] (J) to[densely dashed, postaction=decorate,out=0,in=180](i);
\draw [densely dashed, postaction=decorate](F) to[densely dashed, postaction=decorate,out=0,in=0](g);
\end{scope}
\end{tikzpicture}
}
\end{minipage}
\end{multline*}
This equals $S(\wt{a})\otimes S(\wt{b})\otimes \wt{c}+\wt{a}\otimes \wt{b}\otimes S(\wt{c})$.

Relations (D3) and (D4) come from the boundary of graphs where one tree is a tripod connected by a self-edge. If $a\in H$ is attached to this loop, then in the image, using IHX, will be a short chord with the element $a+S(a)$ attached.

\begin{figure}
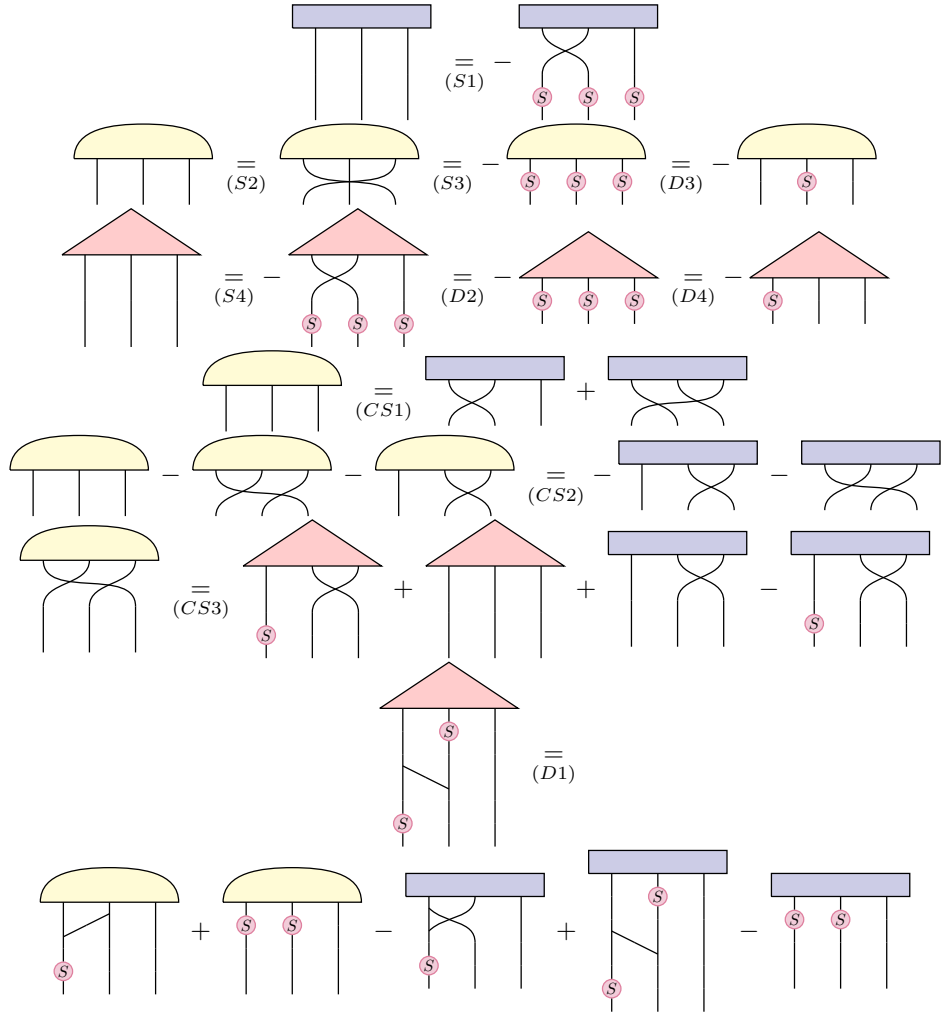

\begin{center}
%\begin{center}
%\begin{tabular}{|c|l|}
%\hline
%(S1)&$a\otimes b\otimes c=-S(b)\otimes S(a)\otimes S(c)$\\
%\hline
%(S2)& $\wt{a}\otimes \wt{b}\otimes \wt{c}=\wt{c}\otimes \wt{b}\otimes \wt{a}$\\
%\hline
%(S3)& $\wt{a}\otimes \wt{b}\otimes \wt{c}=-S(\wt{a})\otimes S(\wt{b})\otimes S(\wt{c})$\\
%\hline
%(S4)& $\wh{a}\otimes \wh{b}\otimes \wh{c}=-S(\wh{b})\otimes S(\wh{a})\otimes S(\wh{c})$\\
%\hline
%(CS1)& $\wt{a}\otimes\wt{b}\otimes\wt{c}=b\otimes c\otimes a+b\otimes a\otimes c$\\
%\hline
%(CS2)& $ \wt{a}\otimes \wt{b}\otimes \wt{c}=\wt{c}\otimes \wt{a}\otimes \wt{b}-a\otimes c\otimes b-c\otimes a\otimes b+\wt{a}\otimes \wt{c}\otimes \wt{b}$\\
%\hline
%(CS3)& $ \wh{a}\otimes \wh{b}\otimes \wh{c}=\wt{c}\otimes \wt{a}\otimes \wt{b}-a\otimes c\otimes b-S(\wh{a})\otimes \wh{c}\otimes \wh{b}+S(a)\otimes c\otimes b$\\
%\hline
%(D1)& $S(\wt{a}_{(1)})\otimes S(\wt{a}_{(2)})\wt{b}\otimes\wt{c}-S(a_{(1)})b\otimes S(a_{(2)})\otimes c-S(\wh{a})\wh{b}_{(1)}\otimes S(\wh{b}_{(2)})\otimes\wh{c}$\\
%&$+S(a)b_{(1)}\otimes S(b_{(2)})\otimes c+S(\wt{a})\otimes S(\wt{b})\otimes \wt{c}-S(a)\otimes S(b)\otimes c=0$ \\
%\hline
%(D2)& $\wh{a}\otimes \wh{b}\otimes \wh{c}=-S(\wh{a})\otimes S(\wh{b})\otimes S(\wh{c})$\\
%\hline
%(D3)&  $\wt{a}\otimes \wt{b}\otimes \wt{c}=-\wt{a}\otimes S(\wt{b})\otimes \wt{c}$\\
%\hline
%(D4)& $\wh{a}\otimes \wh{b}\otimes \wh{c}=-S(\wh{a})\otimes \wh{b}\otimes \wh{c}$\\
%\hline
%\end{tabular}
%\end{center}
\vspace{1em}$\,$\\
$\tp{2cm}{\boxtensor\transF{1}\transF{2}}
\underset{(S1)}{=}-\tp{2cm}{\boxtensor\transA{1}\antip{1}{1}{1}{2}}$\\
$\tp{2cm}{\roundtensor\transF{1}}\underset{(S2)}=\tp{2cm}{\roundtensor\transB{1}}\underset{(S3)}{=}-\tp{2cm}{\roundtensor\antip{1}{1}{1}{1}}
\underset{(D3)}{=}-\tp{2cm}{\roundtensor\antip{0}{1}{0}{1}}$\\
$\tp{2cm}{\tritensor\transF{1}\transF{2}}
\underset{(S4)}{=}-\tp{2cm}{\tritensor\transA{1}\antip{1}{1}{1}{2}}\underset{(D2)}{=}-\tp{2cm}{\tritensor\antip{1}{1}{1}{1}}\underset{(D4)}{=}-\tp{2cm}{\tritensor\antip{1}{0}{0}{1}}$\\
 $\tp{2cm}{\roundtensor\transF{1}}\underset{(CS1)}{=}\tp{2cm}{\boxtensor\transA{1}}+\tp{2cm}{\boxtensor\transE{1}}$\\
$\tp{2cm}{\roundtensor\transF{1}}-\tp{2cm}{\roundtensor\transD{1}}-\tp{2cm}{\roundtensor\transC{1}}\underset{(CS2)}{=}-\tp{2cm}{\boxtensor\transC{1}}-\tp{2cm}{\boxtensor\transD{1}} $\\
 $\tp{2cm}{\roundtensor\transD{1}\transF{2}}\underset{(CS3)}{=}\tp{2cm}{\tritensor\transC{1}\antip{1}{0}{0}{2}}+\tp{2cm}{\tritensor\transF{1}\transF{2}}+\tp{2cm}{\boxtensor\transC{1}\transF{2}}-\tp{2cm}{\boxtensor\transC{1}\antip{1}{0}{0}{2}}$\\
 $\tp{2cm}{\tritensor\antip{0}{1}{0}{1}\opT{1}{0}{2}\antip{1}{0}{0}{3}}\underset{(D1)}{=}\tp{2cm}{\roundtensor\opT{0}{0}{1}\antip{1}{0}{0}{2}}+\tp{2cm}{\roundtensor\antip{1}{1}{0}{1}\transF{2}}-\tp{2cm}{\boxtensor\opX{0}{1}\antip{1}{0}{0}{2}}
+\tp{2cm}{\boxtensor\antip{0}{1}{0}{1}\opT{1}{0}{2}\antip{1}{0}{0}{3}}
-\tp{2cm}{\boxtensor\antip{1}{1}{0}{1}\transF{2}}
$
\end{center}
\caption{Rank 3 relations in pictorial format}\label{fig:relset}
\end{figure}
All of the relations are summarized in Figure~\ref{fig:relset}.
\subsection{Special case: $H=\sym(V)$ in even degree}
The set of relations in Figure~\ref{fig:relset} simplifies quite a bit if we assume that $H=\sym(V)$ is commutative and that the degree is even. In this case $S\otimes S\otimes S$ acts as the identity on $H^{\otimes 3}$, ensuring by relations (S3) and (D2) that the tensors $\wt{a}\otimes\wt{b}\otimes\wt{c}$ and $\wh{a}\otimes\wh{b}\otimes \wh{c}$ are zero. So the above presentation becomes very simple. (S1) and (CS1) (or equivalently (CS2)) imply that the rectangular tensor is totally antisymmetric.
$$\tp{2cm}{\boxtensor\transF{1}}=-\tp{2cm}{\boxtensor\transA{1}}=-\tp{2cm}{\boxtensor\transC{1}}$$
 (CS3) kills any tensors with a factor of odd degree,
 $$\tp{2cm}{\boxtensor\transF{1}}=\tp{2cm}{\boxtensor\antip{1}{0}{0}{1}}$$
  and finally (D1) is equivalent to
$$\tp{2cm}{\boxtensor\transF{1}}=\tp{2cm}{\boxtensor\opT{0}{0}{1}}+\tp{2cm}{\boxtensor\opT{1}{0}{1}}$$
This proves Theorem~\ref{thm:evenpres}.

\subsection{Simplifying the relations}
First note that relation (D1) can be rewritten by appending the invertible automorphism $\tp{1.5cm}{\antip{1}{0}{0}{1}\opT{1}{1}{2}\antip{0}{1}{0}{3}}$ to the bottom, yielding
$$\tp{2cm}{\tritensor\transF{1}}\underset{(D1')}{=}\tp{2cm}{\roundtensor\opX{0}{1}}+\tp{2cm}{\roundtensor\opT{1}{0}{1}}-\tp{2cm}{\boxtensor\opT{0}{0}{1}}+\tp{2cm}{\boxtensor\transF{1}}-\tp{2cm}{\boxtensor\opT{1}{0}{1}}.$$
Now, removing the round tensors using (CS1) gives
$$\tp{2cm}{\tritensor\transF{1}}\underset{(D1'')}{=}\tp{2cm}{\boxtensor\transF{1}}+\tp{2cm}{\boxtensor\transC{1}\opT{0}{0}{2}}+\tp{2cm}{\boxtensor\opX{1}{1}}
+\tp{2cm}{\boxtensor\opX{1}{2}\transC{1}}-\tp{2cm}{\boxtensor\opT{1}{0}{1}}
$$

Substituting into (S2) is tautological. Substituting into (S3) and using (S1) gives
$$\tp{2cm}{\boxtensor\transE{1}}+\tp{2cm}{\boxtensor\transA{1}}\underset{(S3')}{=}\tp{2cm}{\boxtensor\transF{1}}+\tp{2cm}{\boxtensor\transB{1}}.
$$
This is also the result of plugging into (CS2.) This allows us to simplify (D1'') to
$$\tp{2cm}{\tritensor\transF{1}}\underset{(D1''')}{=}\tp{2cm}{\boxtensor\transF{1}}+\tp{2cm}{\boxtensor\transC{1}\opT{0}{0}{2}}+\tp{2cm}{\boxtensor\transD{1}\opX{1}{2}}
$$
This relation and (CS1) then allow us to write all tensors in terms of the rectangular ones $a\otimes b\otimes c$. Rewriting all relations in terms of rectangular ones yields the presentation in Theorem~\ref{thm:rank3pres}.

(S4) becomes a tautology.

(D2) becomes relation (4).

(CS3) becomes relation (6).

As already noted (S1) becomes (1) and S3 becomes (2).

D3 becomes (3) and D4 becomes (5).
\subsection{Proof of Theorem~\ref{thm:oddpres}}
We work with the presentation from Theorem~\ref{thm:rank3pres}. When $H=\sym(V)$, $S\otimes S\otimes S$ acts on the odd degree part of $H^{\otimes 3}$ by $-1$. 

\subsection{Proof of Theorem~\ref{thm:omega3pres}}
When calculating $\Omega_3(H)$, the only boundaries to consider are $D1$ and also $D3$ and $D4$ in the special case where the loop has no element of $H$ on it. So we divide by relations $1\otimes\widehat{a}\otimes \widehat{b}=0=\widetilde{a}\otimes 1\otimes \widetilde{b}$. Converting these to rectangular tensors yields the relations $1\otimes a\otimes b+1\otimes b\otimes a+b\otimes \Delta(a)=0$, and $1\otimes a\otimes b+1\otimes b\otimes a=0$, which together yield the relations \ref{itemA} and \ref{itemB}.

\section{Computer calculations}
Using the presentations given in Theorems~\ref{thm:rank2pres}, \ref{thm:rank3pres}, \ref{thm:omega2pres} and \ref{thm:omega3pres}  computations were made and are listed in Figures~\ref{fig:Rank2} and \ref{fig:Rank3}. 
Note that the data is consistent with Corollaries~\ref{cor:omega2comp},  ~\ref{cor:rank3evencomp}, and Theorem ~\ref{thm:iota}. Note that these computations give slightly different answers than those calculated in \cite{C}. For example in that paper, the decomposition of $\Omega_2(T(V))_4$ was given as $[31]\oplus[1^4]$ which is not correct. It should be $[4]\oplus[31]$, which is consistent with Corollary~\ref{cor:omega2comp}.

The source code for the current computations are available as a mathematica notebook at \texttt{http://www.math.utk.edu/\~{}jconant/} and in the arXiv source folder for this paper. 

\begin{figure}[h]
\begin{tabular}{|c|l|l|l|l|}
\hline
degree&$\homfunctor_2(\sym(V))$&$\homfunctor_2(T(V))$&$\Omega_2(\sym(V))$&$\Omega_2(T(V))$\\
\hline
0&0&0&0&0\\
\hline
1&0&0&0&0\\
\hline
2&0&0&0&0\\
\hline
3&0&0&0&0\\
\hline
4&$[31]$&$[31]$&$[4]\oplus[31]$&$[4]\oplus [31]$\\
\hline
5&0&$[31^2]\oplus[2^21]\oplus[21^3]$&$0$&$2[31^2]\oplus[2^21]\oplus[21^3]$\\
\hline
6&$[51]$ &$[1^6]\oplus[51]\oplus[21^4]$&$[6]\oplus 2[51]$&$[1^6]\oplus 2[51]\oplus 2[21^4]$\\
 & &$\oplus[42]\oplus2[2^21^2]\oplus [3^2]$&$\oplus[42]$&$\oplus 3[42]\oplus 2[2^21^2] \oplus[3^2]$\\
 & &$\oplus[2^3]\oplus 2[321]$ & &$\oplus 2[2^3]\oplus 3[321]\oplus[6]$\\
 \hline
 7&$0$&?&0&?\\
 \hline
8& $[71]\oplus[53]$&?&$2[8]\oplus 2[71]$&?\\
&&&$\oplus2[62]\oplus[53]$&\\
 \hline
\end{tabular}
\caption{Rank 2 Computations}\label{fig:Rank2}
\end{figure}

\begin{figure}[h]
\begin{tabular}{|l|l|l|l|l|}
\hline
degree&$\homfunctor_3(\sym(V))$&$\homfunctor_3(T(V))$&$\Omega_3(\sym(V))$&$\Omega_3(T(V))$\\
\hline
0&0&0&0&0\\
\hline
1&0&0&0&0\\
\hline
2&0&0&0&0\\
\hline
3&$[21]$&$[21]$&$[21]$&$[21]$\\
\hline
4&$0$&${[1^4]}$&$[4]\oplus[31]$&${[4]\oplus[31]\oplus[21^2]\oplus[1^4]}$\\
\hline
5&$[41]\oplus[32]$&$[41]\oplus[32]$&${[31^2]\oplus[2^21]\oplus }3[41]$&${2[5]\oplus 3[41]\oplus[21^3]}$ \\
&&&$\oplus3[32]\oplus2[5]$&${\oplus3[31^2]\oplus 4[32]\oplus 3[2^21]}$\\
\hline
6&[42]&?&$ [6]\oplus2[51]\oplus[42]$&?\\
\hline
7&$[7]\oplus2[61]\oplus2[52]$&?&$5[7]\oplus 8[61]\oplus 8[52]$&?\\
&$\oplus [51^2]\oplus[43]$&&$\oplus 3[51^2]\oplus 5[43]\oplus 4[421]$&\\
&$\oplus[421]\oplus[3^21]$&&$\oplus [3^21]\oplus[32^2]$&\\
\hline
8&$[62]$&?&$2[8]\oplus 2[71]\oplus 2[62]\oplus[53]$&?\\
\hline
\end{tabular}
\caption{Rank 3 Computations}\label{fig:Rank3}
\end{figure}

\end{document}